\documentclass[11pt,a4paper]{article}
\usepackage[english]{babel}
\usepackage{amssymb}
\usepackage{amsmath}
\usepackage{amsthm}
\usepackage[a4paper,margin=2.54cm]{geometry}
\usepackage{graphicx}
\usepackage{todonotes}
\usepackage{comment}

\theoremstyle{plain}
\newtheorem*{theorem*}{Theorem}
\newtheorem{theorem}{Theorem}[section] 
\newtheorem{lemma}[theorem]{Lemma}
\newtheorem{proposition}[theorem]{Proposition}
\newtheorem{corollary}[theorem]{Corollary}
\newtheorem{conjecture}[theorem]{Conjecture}

\theoremstyle{definition}
\newtheorem{definition}[theorem]{Definition}
\newtheorem{remark}[theorem]{Remark}

\makeatletter
\@addtoreset{equation}{section}
\makeatother
\numberwithin{equation}{section}

\DeclareMathOperator{\im}{Im}

\DeclareMathOperator{\Wr}{Wr}

\usepackage{authblk}
\title{Exceptional Laguerre polynomials}

\author[1]{Niels Bonneux}
\author[2]{Arno B.J. Kuijlaars}
\affil[1,2]{Katholieke Universiteit Leuven, Department of Mathematics, 

Celestijnenlaan~200B box 2400, 3001 Leuven, Belgium. 

E-mail:~{\tt niels.bonneux@kuleuven.be} 

{\tt arno.kuijlaars@kuleuven.be}
}
\date{\today}                     
\usepackage{hyperref}
\AtBeginDocument{} 

\begin{document}

\maketitle


\begin{abstract}
The aim of this paper is to present the construction of exceptional
Laguerre polynomials in a systematic way, and to provide new 
asymptotic results on the location of the zeros.
To describe the exceptional Laguerre polynomials we associate them with
two partitions. We find that the use of partitions is an elegant way to 
express these polynomials 
and we restate some of their known properties in terms of partitions. 
We discuss the asymptotic behavior of the regular zeros and the 
exceptional zeros of exceptional Laguerre polynomials as the degree 
tends to infinity.
\end{abstract}


\section{Introduction}\label{sec:introduction}
Laguerre polynomials are one of the three classes of classical orthogonal polynomials
\cite{Szego}, next to Hermite and Jacobi polynomials. These classical orthogonal
polynomials are very well understood. In the past few years, Laguerre polynomials are extended 
to exceptional Laguerre polynomials
\cite{Duran,GomezUllate_Kamran_Milson-09,GomezUllate_Kamran_Milson-10,
GomezUllate_Kamran_Milson-12,Ho_Odake_Sasaki,Liaw_etal,Odake_Sasaki-10,
Odake_Sasaki-09,Sasaki_Tsujimoto_Zhedanov}, also 
sometimes called multi-indexed Laguerre polynomials 
\cite{Grandati_Quesne,Odake-a,Odake-b,Odake_Sasaki-a,Odake_Sasaki-b,Sasaki_Takemura}. 
The remarkable new feature of this generalization is that the exceptional Laguerre
polynomial does not exist for every degree, as was first discovered by
G\'omez-Ullate, Kamran and Milson \cite{GomezUllate_Kamran_Milson-09}.  
Because of this unusual property, these polynomials were called exceptional. 
Remarkably, the exceptional Laguerre 
polynomials may still form a complete orthogonal system \cite{Duran,Duran_Perez}. 
The orthogonality condition only holds for the admissible degrees and 
they are complete in their natural Hilbert space setting. 

Besides exceptional Laguerre polynomials, there are also exceptional Hermite and
exceptional Jacobi polynomials. Recently, Garc\'ia-Ferrero, G\'omez-Ullate and 
Milson classified exceptional orthogonal polynomials \cite{GarciaFerrero_GomezUllate_Milson} and 
it turns out that every system of exceptional orthogonal polynomials is related 
to one of the classical orthogonal polynomials by a sequence of Darboux 
transformations \cite{Crum,Darboux}. Hence, there are only three 
families: exceptional Hermite, exceptional Laguerre and exceptional Jacobi polynomials.
These exceptional polynomial systems appear in quantum mechanical problems as solutions 
in exactly solvable models \cite{Ho,Quesne} or in superintegrable systems \cite{Marquette_Quesne}. 

From a mathematical point of view, there are two main questions. First of all, 
there is the question about the classification of these polynomials. 
This issue is basically solved in \cite{GarciaFerrero_GomezUllate_Milson}. 
Secondly, there is the question about the properties of the exceptional orthogonal
polynomials. In particular, one could wonder whether these polynomials have 
similar properties as their classical counterparts. A noticeable difference is
that orthogonal polynomials only have real zeros, while the exceptional orthogonal
polynomials can have non-real zeros.

The purpose of this paper is twofold. We give an overview of
the construction and classification of exceptional Laguerre polynomials using
the concept of partitions. Next we derive asymptotic results
about the zeros of these polynomials. In many cases, the exceptional Laguerre 
polynomials are a complete set of eigenpolynomials in an appropriate Hilbert 
space where boundary conditions have to be taken into account. Although this is
an important issue, we do not consider it in this paper.

We construct the exceptional Laguerre polynomial via two 
partitions \cite{Andrews} as we feel that this is the most
natural setting for studying these polynomials. 
Such a construction by partitions is already used for exceptional Hermite polynomials, 
see for example \cite{GomezUllate_Grandati_Milson-b,Kuijlaars_Milson}.
In the Laguerre case, we show how a pair of partitions 
can be used to construct exceptional Laguerre polynomials
and we give some of their properties. The properties are not new
as they can be found in various places in the literature with varying notation. 
We find it useful to present it here in a systematic way.
In some of the existing literature, the exceptional Laguerre polynomials are 
classified as type I, type II and type III $X_m$-Laguerre polynomials, see e.g.\ \cite{Ho_Sasaki,Liaw_etal}. This classification captures only very specific 
cases for the exceptional Laguerre polynomial. We show how to rewrite
the $X_m$-Laguerre polynomials in our approach using partitions, see
Section \ref{subsec:Xm}.

The second aim of the paper is to address the question about the asymptotic behavior 
of the zeros  of the exceptional Laguerre polynomial. This asymptotic behavior 
is studied in the Hermite case \cite{Kuijlaars_Milson} where the corresponding partition 
is chosen to be even. We use similar ideas in this paper to obtain the new results for 
the Laguerre case. Our results partially answer the conjecture made in
\cite{Kuijlaars_Milson} that their result in the Hermite case should hold 
in the Laguerre and Jacobi case as well. The behavior of the zeros 
of some exceptional Laguerre polynomials is already studied in the papers 
\cite{GomezUllate_Marcellan_Milson,Ho_Sasaki,Horvath}. In  these papers, 
only $1$-step Darboux  transformations are considered whereas we deal with 
general $r$-step transformations, thereby covering all possible cases.  

The zeros are divided into two classes according to their location. The zeros 
which lie in the orthogonality region are called the regular zeros while the 
others are called the exceptional zeros. Our first result is a lower bound 
for the number of regular zeros. Next, we prove a generalization of the 
Mehler-Heine theorem of Laguerre polynomials. We also show that the counting 
measure of the regular zeros, suitably normalized, converges to the Marchenko-Pastur
distribution. This distribution is also the limiting distribution of the scaled 
zeros of classical Laguerre polynomials. Finally, we prove that each simple zero 
of the generalized Laguerre polynomial attracts an exceptional zero of the 
exceptional Laguerre polynomial. The condition that we need simple zeros 
is probably not a restriction. In the Hermite case, it was conjectured 
by Veselov \cite{Felder_Hemery_Veselov} that the zeros of a Wronskian of a fixed set 
of Hermite polynomials are indeed simple, except possibly at the origin. 
Likewise, we conjecture that the zeros of the generalized Laguerre polynomial 
are simple too. An important difference with Veselov's conjecture is that we 
add a condition in our statement. The zeros are not simple in full generality as we 
have explicit examples where non-simple zeros are obtained.

We organize this paper as follows. In Section \ref{sec:XLP} we introduce the 
generalized and exceptional Laguerre polynomial by associating them with two 
partitions. After defining them properly, we give an alternative proof to derive 
the degree and leading coefficient of these polynomials in Section \ref{sec:DegLcoeff}. 
In Section \ref{sec:ConstructionGLP} we go on to discuss the most general 
construction of generalized Laguerre polynomials. The most general construction
involves four types of eigenfunctions of the Laguerre differential operator,
which could be captured by four partitions. However, there is a procedure
with Maya diagrams to reduce it to only two types. 
We follow \cite{GomezUllate_Grandati_Milson-a} who worked out the reduction procedure
for the Hermite case.
The construction of the exceptional Laguerre polynomials in the most general case 
is treated in Section \ref{sec:ConstructionXLP}. 
In this section we also give the relation between our approach using 
partitions and $X_m$-Laguerre polynomials. 

Finally, we state and prove the new results about the zeros of the 
exceptional Laguerre polynomial in \ref{sec:Zeros-Results} and \ref{sec:Zeros-Proofs}. 
These sections are divided into the results according the regular and exceptional zeros. 
We also state a conjecture of simple zeros in \ref{sec:Zeros-Results} and give some 
remarks about it.


\section{Exceptional Laguerre polynomials}\label{sec:XLP}
We start from Laguerre polynomials to define generalized and exceptional Laguerre
polynomials where we introduce them by use of partitions. After the necessary 
definitions, we state a few known results for these polynomials and give an 
elementary duality property for the partitions. Most of the results of this section 
can be found in \cite{Curbera_Duran,Duran}. 

\subsection{Laguerre polynomials}\label{sec:Laguerre}
The Laguerre polynomial of degree $n$ with parameter $\alpha\in\mathbb{R}$ is given 
by the Rodrigues formula
\begin{equation} \label{eq:LagRod}
	L^{(\alpha)}_n(x)
		= \frac{e^x x^{-\alpha}}{n!} \frac{d^n}{dx^n} \left(e^{-x}x^{n+\alpha}\right).
\end{equation}
For $\alpha > -1$ they are orthogonal polynomials on $[0,\infty)$ with respect
to the positive weight function $x^{\alpha} e^{-x}$,
\begin{equation*}
	\int_{0}^{\infty} L^{(\alpha)}_{n}(x) L^{(\alpha)}_{m}(x) x^{\alpha} e^{-x} dx = 0,
	\qquad \text{for } n \neq m.
\end{equation*}
The Laguerre polynomial $L^{(\alpha)}_n$ is an eigenfunction of the differential operator
\begin{equation}\label{eq:LagDV1}
	y \mapsto xy'' + (\alpha+1-x) y'.
\end{equation}
There are other eigenfunctions of this operator which have a polynomial part and
these are listed in Table \ref{tab:1}. We will use these eigenfunctions to define 
the generalized and exceptional Laguerre polynomial. As we see later, the first 
two types of eigenfunctions are the most relevant ones for us.
\begin{table}[h]
	\begin{center}
		\begin{tabular}{|l | c|}
			\hline
			Eigenfunction & Eigenvalue \\
			\hline
			$L^{(\alpha)}_n(x)$ &  $-n$  \\
			$e^{x} L^{(\alpha)}_n(-x)$ & $n+1+ \alpha$ 	 \\
			$x^{-\alpha} L_n^{(-\alpha)}(x)$ & $-n +\alpha$ \\
			$x^{-\alpha} e^x L_n^{(-\alpha)}(-x)$ & $n+1$ \\	
			\hline
		\end{tabular}
		\caption{Eigenfunctions and eigenvalues of \eqref{eq:LagDV1}.} 
		\label{tab:1}
	\end{center}
\end{table}

We can transform the operator \eqref{eq:LagDV1} into
\begin{equation} \label{eq:LagDV2}
	y \mapsto - y'' +  \left( x^2 + \frac{4\alpha^2-1}{4x^2} \right) y,
\end{equation}	
in the sense that if $y$ is an eigenfunction of \eqref{eq:LagDV1} then
$x^{\alpha+\frac{1}{2}} e^{-\frac{1}{2} x^2} y(x^2)$ is an eigenfunction
of \eqref{eq:LagDV2}.
The corresponding eigenfunctions are given in Table \ref{tab:2} 
and we denote them by $\varphi_n^{(\alpha)}$ and $\psi_n^{(\alpha)}$. 
The Laguerre polynomial $L_n^{(\alpha)}$ itself transforms into
\begin{equation}\label{eq:LagDV3}
	\varphi_n^{(\alpha)}(x) = x^{\alpha+\frac{1}{2}} e^{-\frac{1}{2} x^2} 
	L_n^{(\alpha)}(x^2).
\end{equation}
From $\varphi_n^{(\alpha)}$ we obtain the eigenfunction $\psi_n^{(\alpha)}$ 
of \eqref{eq:LagDV2}  by changing $x$ to $ix$ in \eqref{eq:LagDV3}. 
Two more eigenfunctions are obtained by changing $\alpha$ to $-\alpha$.
We arrive at four sets of eigenfunctions with a polynomial part as shown in 
Table \ref{tab:2}. The operator \eqref{eq:LagDV3} is in the Schr\"odinger
form and it is convenient to apply a Darboux-Crum 
(or a sequence of Darboux) transformation(s) to this form of 
the operator \cite{Crum,Darboux}. 
\begin{table}[h]
	\begin{center}
	\begin{tabular}{|l | c|}
		\hline
		Eigenfunction & Eigenvalue \\
		\hline
		$\varphi_n^{(\alpha)}(x) = e^{-\frac{1}{2}x^2}x^{\alpha+\frac{1}{2}} L^{(\alpha)}_n(x^2)$ & $4n+2(1+\alpha)$	 \\
		$\psi_n^{(\alpha)}(x)=e^{\frac{1}{2}x^2}x^{\alpha+\frac{1}{2}} L^{(\alpha)}_n(-x^2)$ & $-4n-2(1+\alpha)$	 \\
		$\varphi_n^{(-\alpha)}(x) = e^{-\frac{1}{2}x^2}x^{-\alpha+\frac{1}{2}} L^{(-\alpha)}_n(x^2)$ & $4n+2(1-\alpha)$ \\
		$\psi_n^{(-\alpha)}(x)=e^{\frac{1}{2}x^2}x^{-\alpha+\frac{1}{2}} L^{(-\alpha)}_n(-x^2)$ & $-4n-2(1-\alpha)$ \\	
		\hline
	\end{tabular}
	\caption{Eigenfunctions and eigenvalues of \eqref{eq:LagDV2}.} 
	\label{tab:2}
	\end{center}
\end{table}

\subsection{Generalized Laguerre polynomials}
We will now define the generalized Laguerre polynomial. Let us however point out
that the term generalized Laguerre polynomial is also used for the polynomial
\eqref{eq:LagRod} with general parameter $\alpha$, to distinguish it from the 
case $\alpha = 0$
(which is then called the Laguerre polynomial). We use generalized Laguerre
polynomial in analogy with generalized Hermite polynomial as for example in 
\cite{Kuijlaars_Milson}.

For us, the  generalized Laguerre polynomial is basically a Wronskian of 
the eigenfunctions in Table \ref{tab:1} with an appropriate prefactor to make
it a polynomial. A Wronskian of a set of sufficiently
differentiable functions $f_1,\dots,f_r$ is the determinant of the matrix 
$\left( \frac{d^{j-1}}{dx^{j-1}}f_i \right)_{1\leq i,j \leq r}$. We denote it by 
\begin{equation*}
	\Wr[f_1,\dots,f_r] =  
	\det \left( \frac{d^{j-1}}{dx^{j-1}}f_i \right)_{1\leq i,j \leq r}. 
\end{equation*}
A generalized Laguerre polynomial is a Wronskian of $r$ Laguerre polynomials with 
the same parameter $\alpha$, but with distinct degrees 
$n_1 > n_2 > \cdots > n_r \geq 1$. We label the generalized Laguerre polynomial 
by the partition $\lambda = (\lambda_1 \geq \cdots \geq \lambda_r)$ where 
$\lambda_j = n_j - r + j$ for $j = 1, \ldots, r$. Then
\begin{equation}\label{eq:OmegaLa}
	\Omega_{\lambda}^{(\alpha)} := \Wr \left[ L^{(\alpha)}_{n_1}, \ldots, L^{(\alpha)}_{n_r} \right]
\end{equation}
is a polynomial of degree $|\lambda| = \sum_{j=1}^r \lambda_j$. The fact that $|\lambda|$ is the highest degree can be easily seen by looking at the diagonal or 
anti-diagonal. The coefficient to $x^{|\lambda|}$ is non-zero because all 
the positive integers $n_j$ are different. See Lemma \ref{lem:GLP1} 
for a precise formula of the leading coefficient. Note that we follow the convention that
a partition is a weakly decreasing sequence of positive integers, which
is the common one in number theory and combinatorics.

A more general approach is possible if we include other eigenfunctions of the 
operator \eqref{eq:LagDV1}. Since we want to end up with polynomials we only 
include the eigenfunctions from \mbox{Table \ref{tab:1}}. In general we could include 
all four types, but it turns out that it suffices to include only eigenfunctions of 
the first two types, see Section \ref{sec:ConstructionGLP}. We consider two 
partitions $\lambda$ and $\mu$ of lengths $r_1$ and $r_2$ with corresponding 
degree sequences $(n_j)_{j=1}^{r_1}$ and $(m_j)_{j=1}^{r_2}$,
\begin{equation} \label{eq:lambdamu} 
\begin{aligned} 
	\lambda & = (\lambda_1, \ldots, \lambda_{r_1}), 
	& \qquad n_j = \lambda_j + r_1 - j, \ \qquad  j=1, \ldots, r_1, \\
	\mu & = (\mu_1, \ldots, \mu_{r_2}), 
	& \qquad m_j = \mu_j + r_2 - j,  \qquad j=1, \ldots, r_2.
\end{aligned}
\end{equation} 
We write $r=  r_1+r_2$ and
\begin{equation} \label{eq:OmegaLaMu}
	\Omega^{(\alpha)}_{\lambda, \mu} :=  e^{-r_2 x} \cdot
	\Wr\left[f_1, \ldots, f_r \right], 
\end{equation}
where 
\begin{align} 
	f_j(x) & = L^{(\alpha)}_{n_j}(x),  \qquad \qquad j=1, \ldots, r_1,  
	\label{eq:fj1} \\
	f_{r_1+j}(x) & = e^x L^{(\alpha)}_{m_j}(-x), \qquad \ j = 1, \ldots, r_2.
	\label{eq:fj2}
\end{align}
The prefactor $e^{-r_2 x}$ in \eqref{eq:OmegaLaMu} guarantees that 
$\Omega^{(\alpha)}_{\lambda, \mu}$ is a polynomial. We see 
that $\Omega^{(\alpha)}_{\lambda, \emptyset} =\Omega^{(\alpha)}_{\lambda}$.

\begin{definition}\label{def:GLP}
The polynomial $\Omega^{(\alpha)}_{\lambda, \mu}$ defined in \eqref{eq:OmegaLaMu} 
is called the generalized Laguerre polynomial of parameter $\alpha$ associated with partitions $\lambda$ and $\mu$.
\end{definition}
When both partitions are empty, the generalized Laguerre polynomial is 
the constant function $1$. 

\begin{remark}\label{remark:LinInd}
As the Laguerre polynomial, the generalized Laguerre polynomial is defined for 
every $\alpha\in\mathbb{R}$ and the polynomial never vanishes identically,
since the functions $f_1, \ldots, f_r$ from \eqref{eq:fj1}-\eqref{eq:fj2}
are always linearly independent. 
As we can see from Table \ref{tab:1} combined with our definition of 
$\Omega^{(\alpha)}_{\lambda, \mu}$, it is possible that for a particular 
choice of $\alpha$, namely $\alpha=-n_i-m_j-1$ for some $i$ and $j$, 
two eigenvalues coincide. Nevertheless, the eigenfunctions are still linearly
independent as one of them has an exponential part while the other one has not.  
\end{remark}

If we apply a few elementary properties of Laguerre polynomials (which are stated further, see \eqref{app:eq:diff1} and \eqref{app:eq:diff2}), we can express the generalized Laguerre polynomial \eqref{eq:OmegaLaMu} as
\begin{equation} \label{eq:OmegaLaMu2}
	\begin{vmatrix}
		L^{(\alpha)}_{n_1}(x) &\dots &L^{(\alpha)}_{n_{r_1}}(x) & L^{(\alpha)}_{m_1}(-x) & \dots & L^{(\alpha)}_{m_{r_2}}(-x)
		 \\[10pt]
		(-1) L^{(\alpha+1)}_{n_1-1}(x) &\dots & (-1) L^{(\alpha+1)}_{n_{r_1}-1}(x) & L^{(\alpha+1)}_{m_1}(-x) & \dots & L^{(\alpha+1)}_{m_{r_2}}(-x) \\[10pt]
		\vdots &\ddots &\vdots & \vdots &\ddots & \vdots \\[10pt]
		(-1)^{r-1}L^{(\alpha+r-1)}_{n_1-r+1}(x) &\dots &(-1)^{r-1}L^{(\alpha+r-1)}_{n_{r_1}-r+1}(x) & L^{(\alpha+r-1)}_{m_1}(-x) &\dots & L^{(\alpha+r-1)}_{m_{r_2}}(-x)
	\end{vmatrix},
\end{equation}
where we set $L^{(\alpha)}_{n}(x)\equiv 0$ whenever $n \leq -1$. The expression \eqref{eq:OmegaLaMu2} shows that our definition of $\Omega^{(\alpha)}_{\lambda, \mu}$ is the same as Dur\'an's definition \cite[Formula (1.8)]{Duran}. Using this expression, one can easily see that the degree of $\Omega^{(\alpha)}_{\lambda, \mu}$ is 
at most $|\lambda|+\sum_{i=1}^{r_2} m_i$. However in this situation there is 
a fair amount of cancellation and we have the following result from Dur\'an 
\cite[Section 5]{Duran}.

\begin{lemma} \label{lem:GLP1}
Take $\alpha\in\mathbb{R}$. Then $\Omega^{(\alpha)}_{\lambda, \mu}$ is a polynomial of 
degree $|\lambda| + |\mu|$ with leading coefficient given by
\begin{equation*}
	(-1)^{\sum\limits_{i=1}^{r_1}n_i} \cdot \frac{\prod\limits_{1\leq i<j\leq r_1}(n_j-n_i) 
	\prod\limits_{1\leq i<j\leq r_2}(m_j-m_i)}{\prod\limits_{i=1}^{r_1}n_i! 
	\prod\limits_{i=1}^{r_2}m_i!},
\end{equation*}
which is independent of the parameter $\alpha$.
\end{lemma}

Actually, Dur\'an \cite{Duran} first considers exceptional  Meixner polynomials and 
obtains Wronskians of Laguerre polynomials as limits of Casorati determinants of 
Meixner polynomials. The degree  statement for these latter determinants follows 
from a quite general result of Dur\'an and de la Iglesia  \cite[Lemma 3.4]{Duran_Iglesia}.
In Section \ref{sec:DegLcoeff} we give a direct argument of the degree reduction.
Therefore Lemma \ref{lem:GLP1} is the Laguerre case of Proposition \ref{prop:degwr} 
where we use that the leading coefficient of the Laguerre polynomial $L^{(\alpha)}_{n}(x)$ is given by $\frac{(-1)^n}{n!}$. \\

The value at the origin of the generalized Laguerre polynomial 
$\Omega^{(\alpha)}_{\lambda, \mu}$ can be computed explicitly. Dur\'an showed 
that this value is non-zero when the parameter $\alpha$ is not a negative 
integer \cite[Lemma 5.1]{Duran}. By investigating this lemma of Dur\'an, 
we easily verify that the condition can be weakened.

\begin{lemma}\label{lem:GLP2}
Take $\alpha\in\mathbb{R}$ such that the following conditions are satisfied,
\begin{align*}
	\alpha &\neq -1,-2,\dots, -\max\{n_1, m_1\}, &&\\
	\alpha & \neq -n_i-m_j-1, &&i=1,\dots,r_1 \text{ and } j=1,\dots,r_2.
\end{align*}
Then $\Omega^{(\alpha)}_{\lambda, \mu}(0) \neq 0 $.
\end{lemma}

The statement of Lemma \ref{lem:GLP2} is not best possible, but an explicit 
formulation of the values that are not allowed for $\alpha$ is rather complicated 
to write down, and we not pursue this in this paper. 
The condition $\Omega^{(\alpha)}_{\lambda, \mu}(0) \neq 0 $ will play a role 
in Corollary \ref{cor:XLPRegularZeros}.

The two partitions play a similar role in \eqref{eq:lambdamu} 
as is evident from the following  duality property. 
\begin{lemma} \label{lem:GLP4} 
For every $\alpha\in\mathbb{R}$ and partitions $\lambda$ and $\mu$, we have
\begin{equation} \label{lem:GLP4Formula}
	\Omega^{(\alpha)}_{\lambda, \mu}(x) =  (-1)^{\frac{r_1(r_1-1)}{2} + \frac{r_2(r_2-1)}{2}}
	 \Omega^{(\alpha)}_{\mu, \lambda}(-x).
\end{equation}
\end{lemma}
\begin{proof}
We use the following elementary Wronskian properties. Assume 
that $f_1,\dots,f_r$, $g$ and $h$ are sufficiently differentiable. Then
\begin{align}
	\Wr[ g \cdot f_1,\dots, g \cdot f_r] & = \left(g(x)\right)^{r} \cdot 
	\Wr[f_1, \dots, f_r], 
	\label{eq:Wr1} \\
	\Wr[f_1 \circ h,\dots,f_r\circ h](x) & = 
	\left( h'(x)\right)^{\frac{r(r-1)}{2}} \cdot \Wr[f_1,\dots,f_r](h(x)).
	\label{eq:Wr2} 
\end{align}
Then, from \eqref{eq:OmegaLaMu} and \eqref{eq:Wr1} with $g(x) = e^{-x}$, 
we have
\[ \Omega_{\lambda,\mu}^{(\alpha)}
	= e^{r_1 x} \cdot \Wr \left[ e^{-x} f_1, \ldots, e^{-x} f_r \right], \]
where $f_1, \ldots, f_r$ are the functions from \eqref{eq:fj1}-\eqref{eq:fj2}.
By using \eqref{eq:Wr2} with $h(x) = -x$, we obtain
\[ \Omega_{\lambda,\mu}^{(\alpha)}(x)
	= (-1)^{\frac{r(r-1)}{2}} e^{r_1 x} \cdot 
	\Wr \left[ g_{r_2+1}, \ldots, g_{r}, g_1, \ldots, g_{r_2} \right](-x), \]
where $g_j(x) = e^x f_{r_1+j}(-x) = L^{(\alpha)}_{m_j}(x) $ for $j=1, \ldots, r_2$ and
$g_{r_2+j}(x) = e^x f_{j}(-x) = e^x L_{n_j}^{(\alpha)}(-x)$ for $j=1, \ldots, r_1$.
Permuting the first $r_1$ columns with the last $r_2$ columns gives
an extra factor $(-1)^{r_1r_2}$. Therefore, the total factor is
\begin{equation*}
	(-1)^{r_1r_2 + \frac{r(r-1)}{2} } = (-1)^{\frac{r_1(r_1-1)}{2}+ \frac{r_2(r_2-1)}{2}}.
\end{equation*}
Hence, we obtain \eqref{lem:GLP4Formula} 
in view of the definitions \eqref{eq:OmegaLaMu} and \eqref{eq:fj1}-\eqref{eq:fj2}.
\end{proof}

Dur\'an \cite{Duran} gave sufficient conditions for $\Omega^{(\alpha)}_{\lambda, \mu}$ 
to have no zeros on $[0,\infty)$. In a follow-up paper, 
Dur\'an and P\'erez \cite{Duran_Perez} proved that the obtained conditions are also necessary. For their result we need the notion of an even partition.

\begin{definition} 
A partition $\lambda = (\lambda_1, \ldots, \lambda_r)$ with $\lambda_r\geq 1$ is even 
if $r$ is even and $\lambda_{2j-1} = \lambda_{2j}$ for every $j=1, \ldots, \frac{r}{2}$. 
\end{definition}
By convention, when $r=0$, the (empty) partition is even. 

We state the following result for $\alpha > -1$, even though Dur\'an and P\'erez 
obtain a result for  every $\alpha\in\mathbb{R}$.

\begin{lemma}[\cite{Duran,Duran_Perez}] \label{lem:GLP3} 
Let $\alpha > -1$. Then the polynomial $\Omega^{(\alpha)}_{\lambda,\mu}$ has no zeros 
on $[0,\infty)$ if and only if $\lambda$ is an even partition.
\end{lemma}

Suppose $\alpha > -1$. Combining Lemmas \ref{lem:GLP3} and \ref{lem:GLP4} gives us that 
$\Omega^{(\alpha)}_{\lambda,\mu}$ has no zeros on $(-\infty,0]$ if and only if 
$\mu$ is an even partition. Moreover, $\Omega^{(\alpha)}_{\lambda,\mu}$ has no real 
zeros if and only if both $\lambda$ and $\mu$ are even partitions. 

Finally, we state an invariance property which was conjectured in \cite{Duran} 
and proven in \cite{Curbera_Duran}. This property is very conveniently
stated in terms of partitions, as it involves the conjugate partition.
The partition $\lambda'$ is called the conjugate partition of $\lambda$ if
\begin{equation}\label{eq:PartConj}
	\lambda'_i = \#\{j\in\mathbb{N} \mid \lambda_j \geq i\}, 
	\qquad i = 1, \ldots, \lambda_1.
\end{equation}
If we use a graphical representation of a partition by means of a Young diagram,
then the Young diagram of the conjugate partition is obtained by reflection in the
main diagonal, as illustrated in Figure \ref{fig:conjugate}. 

\begin{figure}[h]
\centering
\begin{tikzpicture}
\node[above] at (-4,1/2) {$\lambda=(4,2,1)$};
\foreach \y in {4}
\draw[step=1/2] (-5,0) grid (-5+\y/2,1/2);
\foreach \y in {2}
\draw[step=1/2] (-5,-1/2) grid (-5+\y/2,0);
\foreach \y in {1}
\draw[step=1/2] (-5,-1) grid (-5+\y/2,-1/2);

\draw[->] (-2,-0.25) -- (-1,-0.25);
\node[above] at (-1.5,-0.25) {Conjugation};

\node[above] at (1.25,1/2) {$\lambda'=(3,2,1,1)$};
\foreach \y in {3}
\draw[step=1/2] (0,0) grid (0+\y/2,1/2);
\foreach \y in {2}
\draw[step=1/2] (0,0-1/2) grid (0+\y/2,0);
\foreach \y in {1}
\draw[step=1/2] (0,-1) grid (0+\y/2,-1/2);
\foreach \y in {1}
\draw[step=1/2] (0,-3/2) grid (0+\y/2,-1);
\end{tikzpicture}
\caption{The conjugate partition \label{fig:conjugate}}
\end{figure}
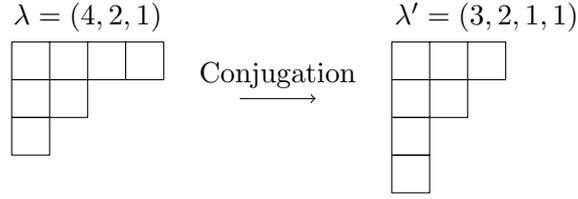

\begin{lemma}[Theorem 6.1 in \cite{Curbera_Duran}] \label{lem:GLP5} 
For every $\alpha\in\mathbb{R}$ and partitions $\lambda$ and $\mu$, we have
\begin{equation*}
	\Omega^{(\alpha)}_{\lambda, \mu}(x) 
		=  (-1)^{\frac{\mu_1(\mu_1 -1)}{2}+|\lambda|+|\mu|+\frac{r_2(r_2-1)}{2}}
		\Omega^{(-\alpha-t)}_{\lambda',\mu'}(-x),
\end{equation*}
where
\[ t = \lambda_1 + \mu_1 + r_1 + r_2, \]
with the convention that $\lambda_1 = 0$ if $r_1 = 0$ (i.e., $\lambda$ is the empty
partition) and $\mu_1 = 0$ if $r_2 = 0$.  
\end{lemma}

Lemma \ref{lem:GLP5} will follow from Theorem \ref{thm:appA} in Section \ref{sec:ConstructionGLP}. The proof is given in Section \ref{sec:ConjPart}.

\subsection{Exceptional Laguerre polynomials}
We fix the parameter $\alpha$ and two partitions $\lambda$ and $\mu$ of lengths 
$r_1$ and $r_2$, respectively. As before we write $r:=r_1+r_2$. Furthermore, 
suppose that the functions $f_1, \ldots, f_r$ are as in \eqref{eq:fj1}-\eqref{eq:fj2}.
We obtain the exceptional Laguerre polynomials by adding one Laguerre polynomial with 
the same parameter $\alpha$ but with degree different from $n_1, \ldots, n_{r_1}$ 
to the Wronskian. If we add the Laguerre polynomial of degree $s$, then this is
\begin{equation} \label{eq:XLPexample}
	e^{-r_2x} \cdot \Wr \left[f_1, \ldots, f_r, L_s^{(\alpha)} \right],
\end{equation}
where we assume that $s \not\in \{n_1, \ldots, n_{r_1}\}$. Up to a possible sign factor,
this polynomial is the generalized Laguerre polynomial of parameter $\alpha$ 
associated with the partitions $\tilde{\lambda}$ and $\mu$, where
$\tilde{\lambda}$ is the partition corresponding to the degrees $n_1,\dots,n_{r_1}$ and $s$.
Since $|\tilde{\lambda}| = |\lambda| + s - r_1$, we have, by Lemma \ref{lem:GLP1}, that
the degree of \eqref{eq:XLPexample} is $s + |\lambda| + |\mu| - r_1$. By varying $s$, 
we obtain polynomials of degrees that are in the degree sequence 
associated with $\lambda$ and $\mu$ that is defined as follows.

\begin{definition}
The degree sequence associated with partitions $\lambda$ and $\mu$ is
\begin{equation} \label{eq:nLaMu} 
	\mathbb N_{\lambda, \mu}
		= \{ n \in \mathbb{N} \cup \{0\} \mid  n\geq |\lambda|+|\mu| - r_1 
		\text{ and } n - |\lambda|-|\mu| \neq \lambda_j-j \text{ for } j=1,\dots,r_1 \}.
\end{equation}
\end{definition}

For $n \in \mathbb N_{\lambda, \mu}$ we take $s = n - |\lambda| - |\mu| + r_1$
and this is a non-negative integer because of the first condition in \eqref{eq:nLaMu}.
The second condition is such that $s\neq n_j$ for every $j=1,\dots,r_1$. 
This then leads to the following definition of the exceptional
Lauerre polynomials.

\begin{definition}
The exceptional Laguerre polynomials of parameter $\alpha$ associated with the two
partitions $\lambda$ and $\mu$ are given by
\begin{equation}\label{def:XLP1}
	L^{(\alpha)}_{\lambda,\mu,n}(x) 
		=  e^{-r_2 x} \cdot \Wr \left[ f_1, \ldots, f_r, L^{(\alpha)}_s \right], 
	\qquad n \in \mathbb N_{\lambda, \mu},
\end{equation}
where $s = n - |\lambda|- |\mu| + r_1$ and $f_1, \ldots, f_r$ are as in \eqref{eq:fj1} and \eqref{eq:fj2}.
\end{definition}

The definition is such that $L^{(\alpha)}_{\lambda,\mu,n}$ has degree $n$. 
There are $|\lambda| + |\mu|$ degrees that do not occur, namely those 
non-negative integers outside of $\mathbb N_{\lambda, \mu}$. They are called the exceptional degrees.
The leading coefficient of $L^{(\alpha)}_{\lambda,\mu,n}$ can be determined 
as in Lemma \ref{lem:GLP1}.

\begin{remark}
When both partitions are empty we obtain that 
$L^{(\alpha)}_{\emptyset,\emptyset,n}(x) = L^{(\alpha)}_n(x)$ for all 
$x\in\mathbb{C}$. Hence, the exceptional Laguerre polynomial is a generalization of 
the usual Laguerre polynomial. 
\end{remark}

\begin{remark}
Similarly to \eqref{def:XLP1}, we can define another exceptional Laguerre polynomial by 
\begin{equation}\label{def:XLP2}
	\tilde{L}^{(\alpha)}_{\lambda,\mu,n}(x) 
	= e^{-(r_2+1)x} \cdot \Wr\left[f_1, \ldots, f_r, e^x L^{(\alpha)}_s(-x) \right],
\end{equation}
where we take $s = n - |\lambda| - |\mu| + r_2 \geq 0$, and $s \neq m_j$ for
$j = 1, \ldots, r_2$. Using the duality from Lemma \ref{lem:GLP4}, 
we can reduce this to the case \eqref{def:XLP1} since
\begin{equation}\label{eq:XLP1+2}
	\tilde{L}^{(\alpha)}_{\lambda,\mu,n}(x)
	= (-1)^{\frac{r_1(r_1+1)}{2} + \frac{r_2(r_2+1)}{2}} L^{(\alpha)}_{\mu,\lambda,n}(-x).
\end{equation}
\end{remark}

If $\alpha>-1$ and the partition $\lambda$ is even, then the exceptional Laguerre
polynomials form a complete set of orthogonal polynomials on the positive real 
line. This result is due to Dur\'an and P\'erez \cite{Duran,Duran_Perez}.

\begin{lemma} \label{lem:ELP}
Suppose $\alpha>-1$. If $\lambda$ is an even partition, then the polynomials $L^{(\alpha)}_{\lambda,\mu,n}$ for $n \in \mathbb N_{\lambda, \mu}$ are orthogonal on $[0,\infty)$ with respect to the positive weight function
\begin{equation*}
	W^{(\alpha)}_{\lambda, \mu}(x) = \frac{x^{\alpha + r} e^{-x}}
	{\left(\Omega^{(\alpha)}_{\lambda, \mu}(x) \right)^2}, \qquad x > 0.
\end{equation*}
That is, if $n, m \in \mathbb N_{\lambda, \mu}$ with $n \neq m$, then
\begin{equation*}
	\int_0^{\infty} 	L^{(\alpha)}_{\lambda, \mu, n}(x) 	L^{(\alpha)}_{\lambda, \mu, m}(x) 
	W^{(\alpha)}_{\lambda, \mu}(x) dx = 0.
\end{equation*}
Moreover, they form a complete orthogonal set in 
$L^2\left([0,\infty), W^{(\alpha)}_{\lambda, \mu}(x)dx\right)$.
\end{lemma}

The conditions $\alpha >-1$ and $\lambda$ an even partition are not best possible.
More general, but less easy to state, conditions can also be found in 
\cite{Duran,Duran_Perez}.

\begin{remark}
In the literature, one often refers to exceptional Laguerre polynomials when these 
polynomials form an orthogonal complete set. In our set-up, we defined a set of polynomials 
in \eqref{def:XLP1} and named them as exceptional Laguerre polynomials for every two 
partitions $\lambda$ and $\mu$ and for every choice of parameter $\alpha$. 
For us, the exceptional Laguerre polynomials form an orthogonal complete set of 
polynomials  only in certain special cases, such as given by Lemma \ref{lem:ELP}.
\end{remark}

\section{The degree and leading coefficient of $\Omega^{(\alpha)}_{\lambda, \mu}$} \label{sec:DegLcoeff}
In this section we give an alternative proof of Lemma \ref{lem:GLP1}. It follows from the subsequent more general result that holds for arbitrary polynomials, and not just for Laguerre polynomials.

\begin{proposition}\label{prop:degwr}
	Let $r_1,r_2$ be non-negative integers and define $r=r_1+r_2$. Let $R_1,\dots,R_r$ be 
	non-zero polynomials such that $\deg R_i \neq \deg R_{j}$ whenever $i\neq j$ and 
	$1\leq i,j \leq r_1$ or $r_1+1 \leq i,j \leq r$. Then the polynomial
	\begin{equation} \label{eq:defOmega}
	\Omega(x) = 
	e^{-r_2x} \cdot \Wr\left[R_1,\dots,R_{r_1},e^{x} R_{r_1+1},\dots,e^{x} R_{r}\right]
	\end{equation}
	has degree
	\begin{equation} \label{eq:degOmega}
	\deg \Omega = \sum\limits_{i=1}^{r} \deg R_{i} - \binom{r_1}{2} - \binom{r_2}{2}.
	\end{equation}
	Moreover, if all polynomials $R_i$ are monic, then the leading coefficient of $\Omega$ 
	is given by
	\begin{equation} \label{eq:lcOmega}
	\prod_{1\leq i < j \leq r_1}\left(\deg R_j - \deg R_i \right) \prod_{1\leq i < j \leq r_2}\left(\deg R_{r_1+j} - \deg R_{r_1+i} \right).
	\end{equation}
\end{proposition}

We use the notation
\begin{equation} \label{eq:defOmega2} 
\Omega =  \Omega(R_1, \ldots R_{r_1}; R_{r_1+1}, \ldots, R_{r})
= e^{-r_2x} \cdot \Wr\left[R_1,\dots,R_{r_1},e^{x} R_{r_1+1},\dots,e^{x} R_{r}\right].
\end{equation}
The proof is by induction on the total sum of the degrees, and
in the induction we use the following lemma.

\begin{lemma} Let $R_1, \ldots, R_r$ be polynomials, not necessarily of distinct degrees.
	Then
	\begin{multline} \label{eq:Omegaprime} 
	\frac{d}{dx}  \Omega(R_1, \ldots R_{r_1}; R_{r_1+1}, \ldots, R_{r}) \\
	= \Omega(R_1', \ldots, R_{r_1}; R_{r_1+1}, \ldots, R_{r}) + \cdots +
	\Omega(R_1, \ldots, R_{r_1}'; R_{r_1+1}, \ldots, R_{r}) \\
	+ \Omega(R_1, \ldots, R_{r_1}; R_{r_1+1}', \ldots, R_{r}) 
	+\cdots +	\Omega(R_1, \ldots, R_{r_1}; R_{r_1+1}, \ldots, R_{r}').
	\end{multline}
\end{lemma}
\begin{proof}
	Since the Wronskian is multilinear in its arguments, we can compute
	from \eqref{eq:defOmega2}
	\begin{align*}
	\frac{d}{dx} \Omega
	=& -r_2 \Omega+ e^{-r_2x} \cdot 
	\frac{d}{dx} \Wr\left[R_1,\dots,R_{r_1},e^{x} R_{r_1+1},\dots,e^{x} R_{r}\right] \\
	=& -r_2 \Omega+
	\Omega(R_1',\dots,R_{r_1}; R_{r_1+1},\dots,R_{r}) + \cdots +
	\Omega(R_1,\dots,R_{r_1}'; R_{r_1+1},\dots, R_{r}) \\
	&+ \Omega(R_1,\dots,R_{r_1};R_{r_1+1} + R_{r_1+1}',\dots, R_{r})  + \cdots
	+ \Omega(R_1,\dots,R_{r_1}; R_{r_1+1},\dots, R_{r} + R_r').
	\end{align*}
	From the multilinearity of $\Omega$ with respect to each of its arguments,
	we then obtain \eqref{eq:Omegaprime}.
\end{proof}

As a second preparation for the proof of Proposition \eqref{prop:degwr}
we state and prove an identity that will be used to establish the formula 
\eqref{eq:lcOmega} for the leading coefficient of $\Omega$. We need the identity \eqref{eq:lem:lcoeff} for natural numbers, but
as the proof shows it is valid for arbitrary real or complex numbers.
\begin{lemma}\label{lem:lcoeff}
	Let $x_1,\dots,x_r$ be $r$ pairwise different numbers, i.e. $x_i\neq x_j$ for $i\neq j$. Then the following holds,
	\begin{equation}\label{eq:lem:lcoeff}
	\sum_{k=1}^{r}\left( x_k \prod_{\shortstack{$\scriptstyle j=1 $ \\ $ \scriptstyle j\neq k$}}^{r} \frac{x_j-(x_k-1)}{x_j-x_k} \right)
	= \sum_{k=1}^{r} x_k - \binom{r}{2}.
	\end{equation}
\end{lemma}	
\begin{proof}
	We prove this by induction on $r$. When $r=1$, the identity is trivially true. So take $r>1$. We claim that
	\begin{equation}\label{eq:lcoeff1}
	\sum_{k=1}^{r} x_k \prod\limits_{\shortstack{$\scriptstyle j=1 $ \\ $ \scriptstyle j\neq k$}}^{r} \frac{x_j-x_k+1}{x_j-x_k}
	= \sum_{k=1}^{r-1} x_k \prod\limits_{\shortstack{$\scriptstyle j=1 $ \\ $ \scriptstyle j\neq k$}}^{r-1} \frac{x_j-x_k+1}{x_j-x_k} + x_r - (r-1)
	\end{equation}
	from which the identity \eqref{eq:lem:lcoeff} follows by applying the induction hypothesis on the right-hand side of the equality. So it remains to show that \eqref{eq:lcoeff1} holds true. We start by splitting the sum into two parts
	\begin{equation}\label{eq:lcoeff2}
	\sum_{k=1}^{r} x_k \prod\limits_{\shortstack{$\scriptstyle j=1 $ \\ $ \scriptstyle j\neq k$}}^{r} \frac{x_j-x_k+1}{x_j-x_k}
	= x_r \prod\limits_{j=1}^{r-1} \frac{x_j-x_r+1}{x_j-x_r} + \sum_{k=1}^{r-1} x_k \prod\limits_{\shortstack{$\scriptstyle j=1 $ \\ $ \scriptstyle j\neq k$}}^{r} \frac{x_j-x_k+1}{x_j-x_k}.
	\end{equation}
	Consider $x_r$ as a variable, we then have the partial fraction decomposition
	\begin{equation}\label{eq:partfracdec}
	x_r \prod_{j=1}^{r-1} \frac{x_j-(x_r-1)}{x_j-x_r}
	= Bx_r+C + \sum_{k=1}^{r-1} \frac{A_k}{x_r-x_k}
	\end{equation}
	for some non-zero constants $B,C$ and
	\begin{equation*}
	A_k = -x_k \frac{\prod\limits_{j=1}^{r-1} (x_j-x_k+1)}{\prod\limits_{\shortstack{$\scriptstyle j=1 $ \\ $ \scriptstyle j\neq k$}}^{r-1}(x_j-x_k)}, \qquad k=1,\dots,r-1.
	\end{equation*}
	From the fact that
	\begin{equation*}
	\frac{x_j-(x_r-1)}{x_j-x_r} 
	= 1 + \frac{1}{x_j-x_r},
	\end{equation*}
	we see that
	\begin{equation*}
	\prod_{j=1}^{r-1} \frac{x_j-(x_r-1)}{x_j-x_r}
	= 1 + \sum_{j=1}^{r-1} \frac{1}{x_j-x_r} + O\left(\frac{1}{x_r^2}\right)
	\end{equation*}
	and thus $B=1$ and $C=-(r-1)$. Therefore, we can rewrite \eqref{eq:partfracdec} as
	\begin{equation}\label{eq:partfracdec2}
	x_r \prod_{j=1}^{r-1} \frac{x_j-(x_r-1)}{x_j-x_r} = -\sum_{k=1}^{r-1} x_k \frac{\prod\limits_{j=1}^{r-1} (x_j-x_k+1)}{\prod\limits_{\shortstack{$\scriptstyle j=1 $ \\ $ \scriptstyle j\neq k$}}^{r}(x_j-x_k)}
	+ x_r - (r-1).
	\end{equation}
	If we now look at the product $\prod\limits_{j=1}^{r-1} (x_j-x_k+1)$, it is clear that this is the same as $j$ runs from $1$ to $r-1$ excluding $k$ as for $j=k$ we have that $x_j+x_k+1=1$. Moreover, as 
	\begin{equation*}
	\frac{1}{x_r-x_k+1}
	= 1 - \frac{x_r-x_k}{x_r-x_k+1},
	\end{equation*}
	by adding $j=r$ in the product, we can rewrite \eqref{eq:partfracdec2} as
	\begin{equation*}
	x_r \prod_{j=1}^{r-1} \frac{x_j-(x_r-1)}{x_j-x_r} 
	= - \sum_{k=1}^{r-1} x_k \prod\limits_{\shortstack{$\scriptstyle j=1 $ \\ $ \scriptstyle j\neq k$}}^{r} \frac{x_j-x_k+1}{x_j-x_k} + \sum_{k=1}^{r-1} x_k \prod\limits_{\shortstack{$\scriptstyle j=1 $ \\ $ \scriptstyle j\neq k$}}^{r-1} \frac{x_j-x_k+1}{x_j-x_k} + x_r - (r-1).
	\end{equation*}
	If we plug this value in \eqref{eq:lcoeff2} we obtain \eqref{eq:lcoeff1} which ends the proof of the lemma.
\end{proof}

Now we are ready for the proof of Proposition \ref{prop:degwr}.

\begin{proof}[Proof of Proposition \ref{prop:degwr}]
	Let $R_1, \ldots, R_r$ be an arbitrary sequence of monic polynomials, and let 
	$\Omega = \Omega(R_1, \dots, R_{r_1}, R_{r_1+1}, \dots, R_r )$ be as in
	\eqref{eq:defOmega}. We are going to prove that
	\begin{equation} \label{eq:degOmegaineq} 
	\deg \Omega \leq \sum_{i=1}^r \deg R_i - \binom{r_1}{2}-\binom{r_2}{2} 
	\end{equation}
	with equality if and only if the degree condition of Proposition \ref{prop:degwr}
	is satisfied. In that case we show that the leading coefficient is
	given by \eqref{eq:lcOmega}.
	If $\Omega \equiv 0$ then we take $\deg \Omega = -\infty$.
	
	If $R_1, \ldots, R_{r_1}$ or $R_{r_1+1}, \ldots, R_r$ are linearly
	dependent then $\Omega \equiv 0$ and then \eqref{eq:degOmegaineq} is
	automatically satisfied. So we assume that $R_1, \ldots, R_{r_1}$
	are linearly independent, as well as $R_{r_1+1}, \ldots, R_r$.
	By permuting entries in the Wronskian we may also assume that
	\begin{equation} \label{eq:Riorder} 
	\begin{aligned} 
	&\deg R_i \leq \deg R_{i+1}, & \qquad \text{ for } i = 1, \ldots, r_1-1, \\
	&\deg R_{r_1+i} \leq \deg R_{r_1+i+1}, & \qquad \text{ for }  i =1, \ldots, r_2-1.
	\end{aligned} 
	\end{equation}
	Another choice would be that the degrees are decreasing as we did so far. However the choice we make does not influence the result of the proposition as the leading coefficient \eqref{eq:lcOmega} captures this choice.
	
	We will then use induction on the number 
	$N=\sum\limits_{k=1}^r \deg R_k - \binom{r_1}{2} - \binom{r_2}{2}$.
	Under the above assumptions the smallest possible number $N= 0$ is reached when
	$\deg R_i = i-1$ for $i =1, \ldots, r_1-1$ and
	$\deg R_{r_1+i} = i-1$ for $i=1, \ldots, r_2 -1$.
	The first $r_1$ columns in the Wronskian \eqref{eq:defOmega} then have
	an upper triangular form and we find by expanding
	\[ \Omega = \left(\prod_{i=1}^{r_1} (\deg R_i)! \right) e^{-r_2x} \cdot 
	\Wr \left[e^x S_1, \dots, e^x S_{r_2} \right] \]
	where $S_j =  e^{-x} \left( \frac{d}{dx} \right)^{r_1} \left(e^x R_{r_1+j} \right)$.
	Note that $S_j$ is a monic polynomial with $\deg S_j = \deg R_{r_1 +j} = j-1$
	for $j=1, \ldots, r_2$.
	Using \eqref{eq:Wr1}, we have 
	\[ e^{-r_2x} \cdot  \Wr \left[e^x S_1, \dots, e^x S_{r_2} \right]
	= \Wr \left[S_1, \dots, S_{r_2} \right] = \prod_{j=1}^{r_2} (\deg R_{r_1+j}) !.
	\]
	Thus
	\begin{equation*}
	\Omega = \prod_{i=1}^{r_1} \left(\deg R_i \right)! \prod_{j=1}^{r_2} \left(\deg R_{r_1+j} \right)!
	\end{equation*}
	which is a constant, so that the degree condition \eqref{eq:degOmega} is satisifed.
	Also the constant is equal to \eqref{eq:lcOmega}, as can be easily verified.
	This completes the proof of the base step of the induction. 
	
	In the induction step we take $N > 0$
	and we assume that the statement is true whenever the sum of the degrees 
	of the polynomials is at most $N-1 + \binom{r_1}{2} + \binom{r_2}{2}$.
	
	We take polynomials $R_1, \ldots, R_r$ with 
	$\sum\limits_{k=1}^r R_k = N + \binom{r_1}{2} + \binom{r_2}{2}$. We assume
	$R_1, \ldots, R_{r_1}$ and $R_{r_1+1}, \ldots, R_{r_2}$ are linearly independent
	and without loss of generality we also assume \eqref{eq:Riorder}.
	If equality holds somewhere in \eqref{eq:Riorder} then we perform a column operation
	on the Wronskian to reduce the degree of one of the polynomials. Then from
	the induction hypothesis it follows that  \eqref{eq:degOmegaineq}  holds 
	with strict inequality. The coefficient of $x^N$ is thus zero, and this 
	agrees with the formula \eqref{eq:lcOmega}.
	
	Thus we may assume
	\begin{equation} \label{eq:Riorder2} 
	\begin{aligned} 
	&\deg R_i < \deg R_{i+1}, & \qquad \text{ for } i = 1, \ldots, r_1-1, \\
	&\deg R_{r_1+i} < \deg R_{r_1+i+1}, & \qquad \text{ for }  i =1, \ldots, r_2-1.
	\end{aligned} 
	\end{equation}
	The identity \eqref{eq:Omegaprime} expresses 
	$\frac{d}{dx}\Omega(R_1, \ldots, R_{r_1}; R_{r_1+1}, \ldots, R_{r_2})$ as as sum of $r$ terms,
	each of which is an $\Omega$-polynomial built out of polynomials whose total
	degree is $N-1 + \binom{r_1}{2} + \binom{r_2}{2}$. According to the induction 
	hypothesis the degree of each of these terms is at most $N-1$. Thus 
	$\deg \Omega' \leq N-1$ and  \eqref{eq:degOmegaineq} follows after an integration step. 
	
	To determine the coefficient of $x^N$, we first compute the coefficient of $x^{N-1}$
	in each of the terms in the right-hand side of \eqref{eq:Omegaprime}.
	We denote $x_k = \deg R_k$ for $k =1, \ldots, r_1$,
	and $y_k = \deg R_{r_1+k}$ for $k = 1, \ldots, r_2$.
	By the induction hypothesis, the coefficient of $x^{N-1}$ of the $k$th term is
	\begin{equation} \label{eq:xkterm} 
	x_k \prod_{j =1 \atop j \neq k}^{r_1} \frac{x_j - (x_k-1)}{x_j-x_k} 
	\prod_{1 \leq i < j \leq r_1} (x_j-x_i) \prod_{1 \leq i < j \leq r_2} (y_j - y_i),
	\qquad k =1, \ldots, r_1, 
	\end{equation}
	while the coefficient for the $(r_1+k)$th term is
	\begin{equation} \label{eq:ykterm} 
	y_k \prod_{j =1 \atop j \neq k}^{r_1} \frac{y_j - (y_k-1)}{y_j-y_k} 
	\prod_{1 \leq i < j \leq r_1} (x_j-x_i) \prod_{1 \leq i < j \leq r_2} (y_j - y_i),
	\qquad k =1, \ldots, r_2. 
	\end{equation}
	Note that \eqref{eq:xkterm} is also valid if $\deg R_1 = 0$, or 
	if $\deg R_k' = \deg R_{k-1}$,   for some $k=2, \ldots, r_1$, since  
	in these cases  \eqref{eq:xkterm} vanishes, as it should.  
	Likewise, \eqref{eq:ykterm} is also valid if $\deg R_{r_1+1} = 0$, or
	if $\deg R_{r_1+k}' = \deg R_{r_1+k}$ for some $k=2, \ldots, r_2$.
	
	Adding \eqref{eq:xkterm} and \eqref{eq:ykterm}, and using Lemma \ref{lem:lcoeff} 
	we find that the coefficient of $x^{N-1}$ in $\Omega'$ is
	\[ \left( \sum_{k=1}^{r_1} x_k - \binom{r_1}{2}
	+ \sum_{k=1}^{r_2} y_k - \binom{r_2}{2} \right)
	\prod_{1 \leq i < j \leq r_1} (x_j-x_i) \prod_{1 \leq i < j \leq r_2} (y_j - y_i).
	\]
	Now recall 
	\[ 
	\sum_{k=1}^{r_1} x_k + \sum_{k=1}^{r_2} y_k
	= \sum_{k=1}^r \deg R_k = N + \binom{r_1}{2} + \binom{r_2}{2}. \]
	Therefore the coefficient of $x^{N-1}$ of $\Omega'$ is
	\[ N  \prod_{1 \leq i < j \leq r_1} (x_j-x_i) \prod_{1 \leq i < j \leq r_2} (y_j - y_i),
	\]
	and after integration we find that the coefficient of $x^N$ of $\Omega$
	is equal to \eqref{eq:lcOmega}. The coefficient is non-zero and therefore
	$\Omega$ has degree $N$ as claimed in \eqref{eq:degOmega}. 
	This completes the proof of the induction step, and Proposition \ref{prop:degwr}
	is proved.
\end{proof}

Using similar ideas we can prove the following proposition
that we will use in Lemma \ref{lem:XLPLinCom}. It applies
to a more general situation, but it only gives an upper bound on
the degree of the polynomial, and it does not give information on the
leading coefficient.

\begin{proposition}\label{prop:degdet}
Take two non-negative integers $r_1,r_2$ and define $r=r_1+r_2$.	
Let $R_1,\dots,R_r$ be polynomials. 
Take non-negative integers $0 \leq l_1<l_2<\dots<l_r$. Consider the polynomial
\begin{equation} \label{eq:defQ}
	Q(x) = e^{-r_2x} 
	\begin{vmatrix}
	R_1^{(l_1)} & \cdots & R_{r_1}^{(l_1)} & \left(e^{x}R_{r_1+1}\right)^{(l_1)} \cdots & \left(e^{x}R_{r}\right)^{(l_1)} \\
	R_1^{(l_2)} & \cdots & R_{r_1}^{(l_2)} & \left(e^{x}R_{r_1+1}\right)^{(l_2)} \cdots & \left(e^{x}R_{r}\right)^{(l_2)} \\
	\vdots & \vdots & \ddots & \vdots \\
	R_1^{(l_r)} & \cdots & R_{r_1}^{(l_r)} & \left(e^{x}R_{r_1+1}\right)^{(l_r)} \cdots & \left(e^{x}R_{r}\right)^{(l_r)} \\
	\end{vmatrix}.
\end{equation}
Then the degree of $Q$ is at most $\sum\limits_{i=1}^{r} \deg R_{i} - \sum\limits_{i=1}^{r_1}l_i - \binom{r_2}{2}$.
\end{proposition}
\begin{proof}
The proof is similar to (part of) the proof of Proposition \ref{prop:degwr}. 
We only need to observe, that similar to \eqref{eq:Omegaprime} we now have
\begin{multline*}
	Q'= Q(R'_1,\dots,R_{r_1};R_{r_1+1},\dots,R_{r}) 
	+ \cdots
	+ Q(R_1,\dots,R'_{r_1};R_{r_1+1},\dots,R_{r}) \\
	+ Q(R_1,\dots,R_{r_1};R'_{r_1+1},\dots,R_{r})
	+ \cdots
	+ Q(R_1,\dots,R_{r_1};R_{r_1+1},\dots,R'_{r})
\end{multline*}
with the (hopefully obvious) notation that
$Q(R_1,\dots,R_{r_1};R_{r_1+1},\dots,R_{r})$ is the polynomial
\eqref{eq:defQ} based on the polynomials $R_1, \ldots, R_r$, and with
the same sequence $l_1 < l_2 < \dots < l_r$. 
\end{proof}

\section{Construction of the generalized Laguerre polynomial}\label{sec:ConstructionGLP}
In this section, we will discuss the construction of the generalized Laguerre polynomial.
In particular, we show that it is sufficient to include only the first two types of
eigenfunctions of Table \ref{tab:1} in the Wronskian. To this end, we start from a
Wronskian including all four types and show that, up ta a constant, this equals 
a Wronskian containing only the first two types as in \eqref{eq:OmegaLaMu}. 
Hence in the general setup we start from eigenfunctions $f_1, \ldots, f_r$
of the Laguerre differential operator \eqref{eq:LagDV1} where 
\begin{align} 
f_j(x) & = L^{(\alpha)}_{n_j}(x), &&j=1,\dots,r_1, \label{app:fj1}\\
f_{r_1+j}(x) & = e^{x}L^{(\alpha)}_{m_j}(-x), &&j=1,\dots,r_2,\label{app:fj2}\\
f_{r_1+r_2+j}(x)&  = x^{-\alpha} L^{(-\alpha)}_{m'_j}(x), &&j=1,\dots,r_3,\label{app:fj3}\\
f_{r_1+r_2+r_3+j}(x) & = e^{x} x^{-\alpha} L^{(-\alpha)}_{n'_j}(-x), &&j=1,\dots,r_4, \label{app:fj4}
\end{align} 
with $r_1 + r_2 + r_3 + r_4 = r$, $n_1 > n_2 > \cdots > n_{r_1} \geq 0$,
$m_1 > m_2 > \cdots > m_{r_2} \geq 0$, $m_1' > m_2' > \cdots > m_{r_3}' \geq 0$,
and $n_1' > n_2' > \cdots > n_{r_1}' \geq 0$.
The result of this section is that there are partitions $\lambda$ and $\mu$,
an integer $t$, and a constant $C$ such that
\begin{equation*}
e^{-(r_2+r_4)x} x^{(\alpha + r_1 + r_2)(r_3+r_4)} \cdot 
\Wr \left[ f_1, f_2, \ldots, f_r \right](x) =
C	\Omega^{(\alpha-t)}_{\lambda, \mu}(x).
\end{equation*}
The fact that we only need two types of eigenfunctions instead of all four possibilities 
is essentially stated in for example \cite{Odake-a,Odake_Sasaki-a,Takemura}. 
However, explicit equalities were never derived. In \cite{Takemura}, Takemura discusses 
the reduction in the Jacobi case and concludes that similar methods must work for the
Laguerre case as well. We derive these explicit identities in this section. We use 
the same ideas to describe this procedure as in the recent paper 
\cite{GomezUllate_Grandati_Milson-a} where the authors discuss 
pseudo-Wronskians of Hermite polynomials.

To describe how $\lambda$, $\mu$ and $t$ are obtained from all the
indices it is useful to introduce Maya diagrams.

\subsection{Maya diagrams}
A Maya diagram $M$ is a subset of the integers that contains a finite number of 
positive integers and excludes a finite number of negative integers. We visualize it as 
an infinite row of boxes which are empty or filled. 
We order these boxes by corresponding them to the set of integers and therefore 
we define an origin. 
To the right of the origin, there are only finitely many filled boxes.
Each of these filled boxes corresponds to a non-negative integer $n\geq0$.
All filled boxes to the right of the origin are labelled by a finite
decreasing  sequence
\[ n_1 > n_2 > \cdots > n_{r_1} \geq 0,  \]
where $r_1$ is the number of filled boxes to the right of the origin. 
If $r_1=0$, then the sequence is empty.

To the left of the origin, there are only finitely many empty boxes.  
Each empty box corresponds to a negative integer $k<0$. We link this negative 
integer to a non-negative integer $n' = -k-1 \geq 0$. We obtain
a second finite decreasing sequence 
\[  n_1' >  n_2' >  \cdots > n_{r_4}' \geq 0,  \]
that labels the positions of the empty boxes to the left of the origin, and $r_4$
is the number of those boxes.  The Maya diagram is encoded by these two sequences
\begin{equation} \label{eq:Mayacoding}
M : \quad \left( n_1', n_2', \ldots, n_{r_4}' \mid n_1, n_2, \ldots, n_{r_1} \right).
\end{equation}

For example, consider the following Maya diagram $M$.
\begin{center}
	\begin{tikzpicture}
	\draw[step=1/2] (-5,0) grid (5,1/2);
	\node[] at (5.5,0.25) {$\dots$};
	\node[] at (5.5,-0.25) {empty boxes};
	\node[] at (-5.5,0.25) {$\dots$};
	\node[] at (-5.5,-0.25) {filled boxes};
	\draw[thick] (0,-1/2)--(0,1);
	\node[] at (0.25,0.75) {0};
	\node[] at (-0.25,0.75) {-1};
	\foreach \y in {0,3,4,5,7}
	\draw[black,fill=black] (\y /2 +0.25,1/4) circle (.5ex);
	\foreach \y in {0,1,4,7,8,9}
	\draw[black,fill=black] (-\y /2 -0.25,1/4) circle (.5ex);
	\end{tikzpicture}
\end{center}
The boxes on the right correspond to the finite decreasing sequence $(7,5,4,3,0)$. 
The empty boxes on the left likewise correspond to $(6,5,3,2)$, and
therefore  
\begin{equation*}
M : \quad \left( 6,5,3,2 \mid 7,5,4,3,0 \right).
\end{equation*}

\begin{remark} \label{rem:Mayalabels} 
	In Theorem \ref{thm:appA}, we extend the decreasing sequence $n_1 > n_2 > \cdots > n_{r_1}$
	to an infinite decreasing sequence $n_1 > n_2 > \cdots > n_{r_1} > n_{r_1+1} > n_{r_1+2} > \cdots$
	giving the positions of all the filled boxes. Thus if $j \geq r_1 + 1$,
	then $n_j < 0$. In our example we get $(7,5,4,3,0,-1,-2,-5,-8,-9,-10,\cdots)$.
\end{remark}

We say that a Maya diagram $\tilde{M}$ is equivalent to $M$ if $\tilde{M}$ is  
obtained from $M$ by moving the position of the origin. Hence, the sequence of 
filled and empty boxes remains unchanged. For example, the Maya diagram $\tilde{M}$
\begin{center}
	\begin{tikzpicture}
	\draw[step=1/2] (-5,0) grid (5,1/2);
	\node[] at (5.5,0.25) {$\dots$};
	\node[] at (5.5,-0.25) {empty boxes};
	\node[] at (-5.5,0.25) {$\dots$};
	\node[] at (-5.5,-0.25) {filled boxes};
	\draw[thick] (0+1.5,-1/2)--(0+1.5,1);
	\node[] at (0.25+1.5,0.75) {0};
	\node[] at (-0.25+1.5,0.75) {-1};
	\foreach \y in {0,3,4,5,7}
	\draw[black,fill=black] (\y /2 +0.25,1/4) circle (.5ex);
	\foreach \y in {0,1,4,7,8,9}
	\draw[black,fill=black] (-\y /2 -0.25,1/4) circle (.5ex);
	\end{tikzpicture}
\end{center}
is equivalent to $M$ as $\tilde{M}=M-3$, i.e., the origin is moved three steps 
to the right. In this example $\tilde{M}$ is given by
\begin{equation*}
\tilde{M} : \quad \left(9,8,6,5,1,0 \mid 4,2,1,0 \right).
\end{equation*}

It is clear from this example that the total length of both sequences
in the encoding of equivalent Maya diagrams need not be the same. 
The total length of $M$ is $9$, while  it is $10$ for $\tilde{M}$.

A canonical choice for the position of the origin is to 
have it such that all boxes to the left are filled, while the first
box to the right is empty. In the example it would be
\begin{center}
	\begin{tikzpicture}
	\draw[step=1/2] (-5,0) grid (5,1/2);
	\node[] at (5.5,0.25) {$\dots$};
	\node[] at (5.5,-0.25) {empty boxes};
	\node[] at (-5.5,0.25) {$\dots$};
	\node[] at (-5.5,-0.25) {filled boxes};
	\draw[thick] (-5+1.5,-1/2)--(-5+1.5,1);
	\node[] at (-5+0.25+1.5,0.75) {0};
	\node[] at (-5-0.25+1.5,0.75) {-1};
	\foreach \y in {0,3,4,5,7}
	\draw[black,fill=black] (\y /2 +0.25,1/4) circle (.5ex);
	\foreach \y in {0,1,4,7,8,9}
	\draw[black,fill=black] (-\y /2 -0.25,1/4) circle (.5ex);
	\end{tikzpicture}
\end{center}
with the encoding
\[ \tilde{M} =  M + 7 : \quad \left( \emptyset \mid 14, 12, 11, 10, 7, 6, 5, 2 \right). \]
Since there are no empty boxes on the left, we only have one decreasing sequence,
say $\left(\emptyset \mid n_1, n_2, \ldots, n_r \right)$ with $n_r \geq 1$.
It is associated with a partition as before, which we denote it by  
\begin{equation} \label{eq:lambdaM} 
\lambda = \lambda(M). 
\end{equation}
In our example, $\lambda(M) = \left( 7,6,6,6,4,4,4,2 \right)$.

We use 
\begin{equation} \label{eq:tM}
t = t(M) 
\end{equation}
to denote the shift $M \to \tilde{M} = M+t$ that we have to apply 
to $M$ to take it into this canonical form.  If $M$ is encoded by \eqref{eq:Mayacoding}
with $r_4 \geq 1$ then
\[ t(M) = n_1' + 1 \]
and in that case $t(M) \geq 1$. If $r_4 =0$, then $t(M) \leq 0$.

Another possible choice for the position of the origin is to have it in such
a way that the number of empty boxes to the left is the same as the number of
filled boxes to the right. It is easy to see that there is a unique such position
and in our example it is 
\begin{center}
	\begin{tikzpicture}
	\draw[step=1/2] (-5,0) grid (5,1/2);
	\node[] at (5.5,0.25) {$\dots$};
	\node[] at (5.5,-0.25) {empty boxes};
	\node[] at (-5.5,0.25) {$\dots$};
	\node[] at (-5.5,-0.25) {filled boxes};
	\draw[thick] (-1+1.5,-1/2)--(-1+1.5,1);
	\node[] at (-1+0.25+1.5,0.75) {0};
	\node[] at (-1-0.25+1.5,0.75) {-1};
	\foreach \y in {0,3,4,5,7}
	\draw[black,fill=black] (\y /2 +0.25,1/4) circle (.5ex);
	\foreach \y in {0,1,4,7,8,9}
	\draw[black,fill=black] (-\y /2 -0.25,1/4) circle (.5ex);
	\end{tikzpicture}
\end{center}
with encoding
\[ M - 1 : \left(7,6,4,3 \mid 6,4,3,2 \right). \]
This coincides with the Frobenius representation of the partition $\lambda(M)$.

We can also put the origin so that all the boxes to the right are empty,
and the first box to the left is filled. In the example it is 
\begin{center}
	\begin{tikzpicture}
	\draw[step=1/2] (-5,0) grid (5,1/2);
	\node[] at (5.5,0.25) {$\dots$};
	\node[] at (5.5,-0.25) {empty boxes};
	\node[] at (-5.5,0.25) {$\dots$};
	\node[] at (-5.5,-0.25) {filled boxes};
	\draw[thick] (2.5+1.5,-1/2)--(2.5+1.5,1);
	\node[] at (2.5+0.25+1.5,0.75) {0};
	\node[] at (2.5-0.25+1.5,0.75) {-1};
	\foreach \y in {0,3,4,5,7}
	\draw[black,fill=black] (\y /2 +0.25,1/4) circle (.5ex);
	\foreach \y in {0,1,4,7,8,9}
	\draw[black,fill=black] (-\y /2 -0.25,1/4) circle (.5ex);
	\end{tikzpicture}
\end{center}
This has an encoding $(n_1', n_2', \ldots, n_s' \mid \emptyset)$ with $n_s' \geq 1$. 
Then 
\[ \lambda'_j = n_j' - s + j, \qquad j =1, \ldots, s \]
is the partition that is conjugate to $\lambda$, see \eqref{eq:PartConj}.

This is all we need about Maya diagrams. A more thorough introduction to Maya diagrams can be found in for example \cite{GomezUllate_Grandati_Milson-a}. More information about the Frobenius representation is covered in \cite{Olsson}.

\subsection{The result} 
Now we are able to state the main result of this section about the construction of the generalized Laguerre polynomial. 

\begin{theorem}\label{thm:appA}
	Let $f_1, \ldots, f_r$ be as in  \eqref{app:fj1}-\eqref{app:fj4}, and let 
	\begin{equation} \label{eq:M1M2}
	\begin{aligned}
	M_1: & \quad \left( n'_1,\dots,n'_{r_4} \mid n_1,\dots,n_{r_1} \right), \\
	M_2: & \quad \left( m'_1,\dots,m'_{r_3} \mid m_1,\dots,m_{r_2} \right),
	\end{aligned}
	\end{equation} 
	be two Maya diagrams built out of the degrees of the Laguerre polynomials
	appearing in \eqref{app:fj1}-\eqref{app:fj4}. Let $\lambda = \lambda(M_1)$
	and $\mu = \lambda(M_2)$ be the two partitions that are associated with $M_1$ and $M_2$
	as described above, and let $t_1 = t(M_1)$, $t_2 = t(M_2)$ as in \eqref{eq:tM}.
	
	Then there is a constant $C = C(\alpha, M_1, M_2)$ such that
	\[  e^{-(r_2+r_4)x} x^{(\alpha + r_1+r_2)(r_3+r_4)} \cdot 
	\Wr \left[f_1, \ldots, f_r \right](x)
	= C \Omega^{(\alpha-t_1-t_2)}_{\lambda, \mu}(x) \]
	where the constant $C$ is given by
	\begin{multline} 
	C(\alpha, M_1, M_2) = (-1)^{d}
	\prod_{j=1}^{r_1} \prod_{k=1}^{r_3} (m_k' - \alpha - n_j) 
	\prod_{j=1}^{r_2} \prod_{k=1}^{r_4} (n_k' - \alpha - m_j) \\
	\times  \label{eq:constantC}
	\prod_{j=1}^{r_1} \prod_{k=1}^{r_4} (n_j + n_k' +1) 
	\prod_{j=1}^{r_2} \prod_{k=1}^{r_3} (m_j + m_k' + 1),
	\end{multline}
	and
	\begin{equation}\label{eq:defd}
	d=
	\begin{cases}
	d_0 &  \text{ if } r_3=r_4=0, \\
	d_0 + d_1  & \text{ if } r_3=0 \text{ and } r_4\geq 1, \\
	d_2  & \text{ if } r_3\geq 1 \text{ and } r_4=0, \\
	d_1+d_2 & \text{ if } r_3\geq 1 \text{ and } r_4\geq 1.
	\end{cases}
	\end{equation}
	with
	\begin{align*}
	&d_0 = |t_2| r_2 - \frac{|t_2|(|t_2|+1)}{2}, 
	\\
	&d_1 = \frac{r_4(r_4-1)}{2} + r_1 r_4 + \frac{(n'_1+1-r_4)(n'_1+2+r_4)}{2} + \sum\limits_{i=1}^{n'_1+1-r_4} n_{r_1+i}, \\
	&d_2 = r_2(m'_1+1) + (m'_1+1)(m'_1+1-r_3) + 
	\sum\limits_{i=1}^{m'_1+1-r_3} m_{r_2+i}.
	\end{align*}
	Here we used the infinite decreasing sequences $(n_i)_{i=1}^{\infty}$ and $(m_i)_{i=1}^{\infty}$ which give 
	the positions of all filled boxes, see Remark \ref{rem:Mayalabels}.
\end{theorem}

The constant \eqref{eq:constantC} is non-zero, if and only if
\begin{align}\label{eq:alphavanishing}
\begin{split}
\alpha \neq m'_i-n_j, \qquad \text{for } i=1,\dots,r_3, \quad j=1,\dots,r_1, \\
\alpha \neq n'_i-m_j,	\qquad \text{for } i=1,\dots,r_4, \quad j=1,\dots,r_2.
\end{split}
\end{align}
Indeed, if $\alpha$ does not satisfy these conditions, $\Omega_{M_1,M_2}^{(\alpha)}$
vanishes identically since in such a case, two functions in the Wronskian 
are multiples of each other and therefore the Wronskian is zero.
If \eqref{eq:alphavanishing} is satisfied then $\Omega^{(\alpha)}_{M_1, M_2}$ 
is a polynomial of degree $|\lambda(M_1)|+|\lambda(M_2)|$.

We prove this theorem as follows. Firstly we show that there is a 
reduction possible when $0$ appears in the encoding of one of the Maya diagrams. 
Secondly we describe the shifting process to reduce $M_1$ to $M_1 + t_1$ and 
$M_2$ to $M_2 + t_2$. This is handled in the upcoming two subsections.

\begin{remark}
In \eqref{def:GLP}, we defined the generalized Laguerre polynomial using the first two 
types of eigenfunctions of Table \ref{tab:1}. Theorem \ref{thm:appA} tells us that if we 
use the four types of eigenfunctions, we can always reduce to the first two types of
eigenfunctions. 

One could make other choices, and shift the Maya diagrams in another way. Then the 
generalized Laguerre polynomial is reduced to only two types of eigenfunctions 
which are not necessarily the first two types as in \eqref{def:GLP}. For example
it is possible to reduce to type one and type three.  
However, one type of eigenfunctions should be from type one  \eqref{app:fj1} or
type four \eqref{app:fj4},  while the other should be from type two \eqref{app:fj2} or
type three \eqref{app:fj3}. The reductions to other combinations  
can be derived in the same way as Theorem \ref{thm:appA}.
\end{remark}

\subsection{Reduction procedure}
We use the notation
\begin{equation}\label{eq:OmegaGeneral}
\Omega^{(\alpha)}_{M_1, M_2}(x) =
e^{-(r_2+r_4)x} x^{(\alpha + r_1+r_2)(r_3+r_4)} \cdot \Wr \left[f_1, \ldots, f_r \right](x)
\end{equation}
where $f_1, \ldots, f_r$ are as in \eqref{app:fj1}-\eqref{app:fj4}, as before.
The following elementary properties of the derivatives will be used in 
the  proof of Lemma \ref{app:lem:1}
\begin{align}
\frac{d}{dx} \left( L^{(\alpha)}_{n}(x) \right) &= - L^{(\alpha+1)}_{n-1}(x), \label{app:eq:diff1} \\
\frac{d}{dx} \left( e^{x} L^{(\alpha)}_{n}(-x) \right) &= e^{x} L^{(\alpha+1)}_{n}(-x), \label{app:eq:diff2} \\
\frac{d}{dx} \left( x^{-\alpha} L^{(-\alpha)}_{n}(x)\right) &= (n-\alpha) x^{-\alpha-1} L^{(-\alpha-1)}_{n}(x), \label{app:eq:diff3} \\
\frac{d}{dx} \left( x^{-\alpha} e^{x} L^{(-\alpha)}_{n}(-x)\right) &= (n+1) x^{-\alpha-1} e^{x} L^{(-\alpha-1)}_{n+1}(-x). \label{app:eq:diff4}
\end{align}

A reduction is possible if $0$ appears in the encoding of one of the Maya diagrams.

\begin{lemma}\label{app:lem:1}
	Let $M_1$ and $M_2$ be given by \eqref{eq:M1M2}.
	\begin{enumerate}
		\item[\rm (a)] If $n_{r_1}=0$, then 
		\begin{equation} \label{eq:reduction1}
		\Omega^{(\alpha)}_{M_1,M_2}
		= \left( \prod_{i=1}^{r_3} (m'_i-\alpha) \prod_{i=1}^{r_4} (n'_i+1) \right) \Omega^{(\alpha+1)}_{M_1-1,M_2}.
		\end{equation}
		\item[\rm (b)] If $n'_{r_4}=0$, then
		\begin{equation} \label{eq:reduction4}
		\Omega^{(\alpha)}_{M_1,M_2}
		= (-1)^{r_1+r_2+r_4-1} \left( \prod_{i=1}^{r_1} (n_i+1)  
		\prod_{i=1}^{r_2} (m_i+\alpha) \right) \Omega^{(\alpha-1)}_{M_1+1,M_2}.
		\end{equation}
		\item[\rm (c)]  If $m_{r_2} = 0$, then
		\begin{equation} \label{eq:reduction2}
		\Omega^{(\alpha)}_{M_1,M_2}
		= (-1)^{r_2-1} \left( \prod_{i=1}^{r_3} (m_i'+1)  
		\prod_{i=1}^{r_4} (n_i' -\alpha) \right) \Omega^{(\alpha+1)}_{M_1,M_2-1}.
		\end{equation}
		\item[\rm (d)]  If $m'_{r_3} = 0$, then
		\begin{equation} \label{eq:reduction3}
		\Omega^{(\alpha)}_{M_1,M_2}
		= (-1)^{r_1+r_2} \left( \prod_{i=1}^{r_1} (n_i+\alpha)  
		\prod_{i=1}^{r_2} (m_i+1) \right) \Omega^{(\alpha-1)}_{M_1,M_2+1}.
		\end{equation}
	\end{enumerate}
\end{lemma}
\begin{proof}
	We give the proof of part (b). The other proofs are similar or, 
	in case of part (a), even simpler.
	
	Assume $n'_{r_4} = 0$. Then $f_r(x) = e^x x^{-\alpha}$.
	We take the factor $e^x x^{-\alpha}$ out of the Wronskian by means of the 
	identity \eqref{eq:Wr1}, and we expand the Wronskian determinant along the
	last row, to obtain
	\begin{align*} \Omega^{(\alpha)}_{M_1,M_2}(x)
	& = \left(e^{x} x^{-\alpha} \right)^r
	e^{-(r_2+r_4)x} x^{(\alpha+r_1+r_2)(r_3+r_4)} \cdot 
	\Wr \left[ e^{-x} x^{\alpha} f_1, \ldots,
	e^{-x} x^{\alpha} f_{r-1}, 1 \right] \\
	& = (-1)^{r-1} \left(e^{x} x^{-\alpha} \right)^r
	e^{-(r_2+r_4)x} x^{(\alpha+r_1+r_2)(r_3+r_4)} \\
	& \qquad \times
	\Wr \left[ \frac{d}{dx} \left(e^{-x} x^{\alpha} f_1\right), \ldots,
	\frac{d}{dx} \left(e^{-x} x^{\alpha} f_{r-1} \right) \right].
	\end{align*}
	We use the identities
	\begin{align*} 
	\frac{d}{dx} \left(e^{-x} x^{\alpha} f_{j} \right) & = 	
	\frac{d}{dx} \left(e^{-x} x^{\alpha} L_{n_j}^{(\alpha)}(x) \right) \\
	& = (n_j+1) x^{\alpha-1} e^{-x} L_{n_j+1}^{(\alpha-1)}(x), && j=1, \ldots, r_1, \\
	\frac{d}{dx} \left(e^{-x} x^{\alpha} f_{r_1 +j} \right)
	& = \frac{d}{dx} \left( x^{\alpha} L_{m_j}^{(\alpha)}(-x) \right) \\
	& = (m_j+\alpha) x^{\alpha-1} L_{m_j}^{(\alpha-1)}(-x), && j=1, \ldots, r_2, \\
	\frac{d}{dx} \left(e^{-x} x^{\alpha} f_{r_1 +r_2 +j} \right) 
	& = \frac{d}{dx} \left( e^{-x}  L_{m'_j}^{(-\alpha)}(x) \right) \\
	& = -e^{-x} L_{m'_j}^{(-\alpha+1)}(x), && j=1, \ldots, r_3, \\
	\frac{d}{dx} \left(e^{-x} x^{\alpha} f_{r_1 +r_2 + r_3 +j} \right)
	& = \frac{d}{dx} \left(L_{n'_j}^{(-\alpha)}(-x) \right) \\
	& = L_{n'_j-1}^{(-\alpha+1)}(-x), && j =1, \ldots, r_4-1,
	\end{align*}
	that follow from  \eqref{app:eq:diff1}-- \eqref{app:eq:diff4}.
	
	Let $\tilde{M}_1 = M_1+1$, and let $\tilde{f}_1, \ldots, \tilde{f}_{\tilde{r}}$
	be the functions associated with Maya diagram $\tilde{M}_1$ and 
	parameter $\tilde{\alpha} = \alpha-1$. We also use $\tilde{r}_1 = r_1$,
	$\tilde{r}_2 = r_2$, $\tilde{r}_3 = r_3$, $\tilde{r}_4 = r_4 - 1$, 
	and $\tilde{r}= r-1$. The above identities show that
	\[ 	\frac{d}{dx} \left(e^{-x} x^{\alpha} f_{j} \right)  =
	C_j e^{-x} x^{\alpha-1}  \tilde{f}_j, \qquad j = 1, \ldots, r-1. \]
	where 
	$C_j = n_j+1$, for $j=1, \ldots, r_1$, 
	$C_{r_1+j} = m_j + \alpha$,  for $j = 1, \ldots, r_2$,
	$C_{r_1+r_2+j} = -1$ for $j=1, \ldots, r_3$, and 	
	$C_{r_1+r_2+r_3+j}  = 1$ for $j=1, \ldots, r_4-1$. 
	
	We take the constant factor $C_j$ out of the $j$th column of the Wronskian,
	for each $j=1, \ldots, r$, and then
	take out the common factor $e^{-x} x^{\alpha-1}$ by means of the identity
	\eqref{eq:Wr1}. This leads to a prefactor $\left(e^{-x} x^{\alpha-1}\right)^{r-1}$
	and so
	\begin{align*} 
	\Omega^{(\alpha)}_{M_1,M_2}(x)
	& = c_1 
	\left(e^x x^{-\alpha}\right)^r
	\left(e^{-x} x^{\alpha-1}\right)^{r-1}  
	e^{-(r_2+r_4)x} x^{(\alpha+r_1+r_2)(r_3+r_4)} \cdot
	\Wr \left[ \tilde{f}_1, \ldots,
	\tilde{f}_{r-1}  \right] \\
	& = c_1
	e^{-(r_2+r_4-1)x} x^{(\alpha+r_1+r_2)(r_3+r_4) - r+1} \cdot
	\Wr \left[ \tilde{f}_1, \ldots,
	\tilde{f}_{r-1}  \right]
	\end{align*}
	where
	\begin{equation*}
	c_1 = (-1)^{r-1} \prod_{j=1}^{r_4-1} C_j = (-1)^{r_1+r_2+r_4 -1} 
	\prod_{j=1}^{r_1} (n_j+1) \prod_{j=1}^{r_2} (m_j +\alpha).
	\end{equation*}
	Hence, \eqref{eq:reduction4} follows since
	\[ (\alpha+r_1+r_2)(r_3+r_4) - r+1 = 
	(\alpha-1 + r_1 + r_2)(r_3+r_4-1)
	= (\tilde{\alpha} + \tilde{r}_1 + \tilde{r}_2)(\tilde{r}_3 + \tilde{r}_4). \]
\end{proof}

\subsection{Shifting process}
Firstly, assume that $t_1 \leq 0$ and $t_2 \leq 0$. Then $r_3=r_4=0$, $r_1 \geq |t_1|-1$ with $n_{r_1-j} = j$, for $j=0, \ldots, |t_1|-1$ and $r_2 \geq |t_2|-1$ with $m_{r_2-j} = j$, for $j=0, \ldots, |t_2|-1$. Therefore we can apply \eqref{eq:reduction1} $|t_1|$ times and \eqref{eq:reduction3} $|t_2|$ times. We find that
\begin{equation*}
\Omega_{M_1,M_2}^{(\alpha)} = C(\alpha) \Omega_{M_1+t_1,M_2+t_2}^{(\alpha-t_1-t_2)}, 
\end{equation*}
where
\begin{equation}\label{eq:Omegashift0}
C(\alpha) = (-1)^{\sum_{i=1}\limits^{|t_2|}r_2-i} = (-1)^{d_0}.
\end{equation}

Secondly, assume that $t_1 > 0$ and denote
\[ \tilde{M_1} = M_1 + 1 : \quad \left(\tilde{n}_1', \dots, \tilde{n}_{\tilde{r}_4}' \mid 
\tilde{n}_1, \dots, \tilde{n}_{\tilde{r}_1} \right). \]
There are two situations, depending on whether the box immediately to the
left of the origin in $M_1$ is filled or not.
\begin{itemize}
	\item
	If it is empty then $n_{r_4}' = 0$, $\tilde{r}_4 = r_4 - 1$, $\tilde{r}_1 = r_1$,
	and 
	\begin{equation} \label{eq:M1tilde1} 
	\tilde{M_1} : \quad \left(n_1'-1, \dots, n_{r_4-1}'-1 \mid 
	n_1+1, \dots, n_{r_1}+1 \right). 
	\end{equation}
	\item
	If it is filled then $\tilde{n}_{r_1+1}=0$, 
	$\tilde{r}_4 = r_4$ $\tilde{r}_1 = r_1 + 1$,	and 
	\begin{equation} \label{eq:M1tilde2}
	\tilde{M_1} : \quad \left(n_1'-1, \dots, n_{r_4}'-1 \mid 
	n_1+1, \dots, n_{r_1}+1, 0 \right).
	\end{equation}
\end{itemize}
In both cases a reduction procedure is possible, since either $n_{r_4}'=0$
or $\tilde{n}_{r_1+1} = 0$. If $n_{r_4}' = 0$, we use \eqref{eq:reduction4}
and if $\tilde{n}_{r_1+1} = 0$, we use \eqref{eq:reduction1} but with
$M_1+1$ instead of $M_1$ and $\alpha-1$ instead of $\alpha$, i.e., 
\begin{equation} \label{eq:reduction5} 
\Omega_{M_1, M_2}^{(\alpha)} = 
\left(	\prod_{i=1}^{r_3} (m_i' - \alpha+1) \prod_{i=1}^{r_4} n_i' \right)^{-1} 
\Omega_{M_1+1, M_2}^{(\alpha-1)}. 
\end{equation}
Continuing in this way, we go from $M_1$ to $M_1 + t_1$, to arrive at 
a Maya diagram with no empty boxes on the left. We find
\begin{equation*}
\Omega_{M_1,M_2}^{(\alpha)} = C_1(\alpha) \Omega_{M_1+t_1,M_2}^{(\alpha-t_1)} 
\end{equation*}
for some constant $C_1(\alpha)$. Now we perform a second step where we shift $M_2$ to $M_2 + t_2$. Hence in a similar way we find for some constant $C_2(\alpha)$
\begin{equation*}
\Omega_{M_1+t_1,M_2}^{(\alpha-t_1)} = C_2(\alpha) \Omega_{M_1+t_1,M_2+t_2}^{(\alpha-t_1-t_2)}.
\end{equation*}
Hence
\[ \Omega_{M_1,M_2}^{(\alpha)} = C_1(\alpha)C_2(\alpha) \Omega_{M_1 + t_1, M_2+t_2}^{(\alpha-t_1-t_2)}, \]
which means in terms of the partitions $\lambda$ and $\mu$ that indeed 
\begin{equation}\label{eq:Omegashift}
\Omega_{M_1,M_2}^{(\alpha)} = C(\alpha) \Omega_{\lambda, \mu}^{(\alpha-t_1-t_2)}, 
\end{equation}
where $C(\alpha)=C_1(\alpha)C_2(\alpha)$.

Thirdly, when $t_2>0$ we first shift the Maya diagram $M_2$ to $M_2+t_2$. Next we transpose $M_1$ to $M_1+t_1$ and we find that \ref{eq:Omegashift} also holds.

The constant $C(\alpha)$ can be found by keeping track of all the prefactors we
find along the way. There is a fair amount of cancellation taking place. In particular
every factor that we would find in the denominator also appears somewhere
in a numerator. It leads to the formula \eqref{eq:constantC}. However, knowing \eqref{eq:constantC}, we can also verify it by induction
on the value of $T=|t_1| + |t_2|$.  

If $t_1=t_2=0$, then $r_3=r_4 = 0$, and all products in \eqref{eq:constantC} 
are empty and $d_0=0$. We then find $C(\alpha)  = 1$ as it should be. 

Take $T > 0$ and assume the formula \eqref{eq:constantC} is correct 
whenever $|t_1| + |t_2| = T-1$. Assuming $|t_1| + |t_2| = T$,
we have a number of cases to consider. Note that we may assume that $t_1>0$ or $t_2>0$ as the situation where $t_1,t_2\leq0$ is derived in \eqref{eq:Omegashift0} which also covers $t_1=t_2=0$.
\begin{itemize}
	\item[(a)] $t_1 > 0$ and $\tilde{M}_1 = M_1 + 1$ is given by \eqref{eq:M1tilde1},
	\item[(b)] $t_1 > 0$ and $\tilde{M}_1 = M_1 + 1$ is given by \eqref{eq:M1tilde2},
	\item[(c)] $t_2 > 0$ and $\tilde{M}_2 = M_1 + 1$ is similar as \eqref{eq:M1tilde1},
	\item[(d)] $t_2 > 0$ and $\tilde{M}_2 = M_1 + 1$ is similar as \eqref{eq:M1tilde2}.
\end{itemize}
As the ideas are similar for each case, we only treat case (b). In that case we have \eqref{eq:reduction5} and so
\begin{equation} \label{eq:C}
C(\alpha,M_1,M_2)
= C(\alpha-1,M_1+1,M_2) 
\left(	\prod_{i=1}^{r_3} (m_i' - \alpha+1) \prod_{i=1}^{r_4} n_i' \right)^{-1}.
\end{equation}

By the induction hypothesis,
\begin{multline}  \label{eq:CM1tilde2}
C(\alpha-1,M_1+1,M_2) 
= (-1)^{\tilde{d}} 
\prod_{j=1}^{\tilde{r}_1} \prod_{k=1}^{r_3} 
\left(m_k' - (\alpha-1) - \tilde{n}_j\right) 
\prod_{j=1}^{r_2} \prod_{k=1}^{\tilde{r}_4}
\left( \tilde{n}_k'-(\alpha-1) - m_j \right) \\
\times
\prod_{j=1}^{\tilde{r}_1}  \prod_{k=1}^{r_4} 
\left(\tilde{n}_j+\tilde{n}'_k + 1 \right)
\prod_{j=1}^{r_2} \prod_{k=1}^{r_3} (m_j + m_k' + 1)
\end{multline}
with the parameters $\tilde{r}_1 = r_1 + 1$, $\tilde{r}_4 = r_4$, $\tilde{n}_j = n_j+1$ for $j=1, \ldots, r_1$, 
$\tilde{n}_{r_1+1} = 0$, and $\tilde{n}'_j = n'_j - 1$ for $j=1, \ldots, r_4$,
as they are associated with $\tilde{M}_1$, see \eqref{eq:M1tilde2}. The first three double products appearing in \eqref{eq:CM1tilde2} are 
\begin{align*}
\prod_{j=1}^{\tilde{r}_1} \prod_{k=1}^{r_3} (m_k'-(\alpha-1)-\tilde{n}_j) 
& = \prod_{j=1}^{r_1} \prod_{k=1}^{r_3} (m_k'-\alpha-n_j) \prod_{k=1}^{r_3} (m'_k-(\alpha-1)), \\
\prod_{j=1}^{r_2} \prod_{k=1}^{\tilde{r}_4} (\tilde{n}_k'-(\alpha-1)-m_j)
& = \prod_{j=1}^{r_2} \prod_{k=1}^{r_4} (n_k'-\alpha-m_j),\\
\prod_{j=1}^{\tilde{r}_1} \prod_{k=1}^{\tilde{r}_4} (\tilde{n}_j + \tilde{n}_k' +1)
& = \prod_{j=1}^{r_1} \prod_{k=1}^{r_4} (n_j + n_k' +1) \prod_{k=1}^{r_4} n_k'.
\end{align*}
Combining this with \eqref{eq:C} and \eqref{eq:CM1tilde2} we obtain
\eqref{eq:constantC} up to the sign. So far, we only have to proof that $(-1)^{\tilde{d}}=(-1)^{d}$. The number $\tilde{d}$ depends on $\tilde{M}_1$ and $M_2$ where $d$ depends on $M_1$ and $M_2$. Trivially, $\tilde{d}_0=d_0$ and $\tilde{d}_2=d_2$ as they are independent of the first Maya diagram. The induction hypothesis gives us
\begin{align*}
\tilde{d}_1
&= \frac{\tilde{r}_4(\tilde{r}_4-1)}{2} + \tilde{r}_1 \tilde{r}_4 + \frac{(\tilde{n}'_1+1-\tilde{r}_4)(\tilde{n}'_1+2+\tilde{r}_4)}{2} + \sum\limits_{i=1}^{\tilde{n}'_1+1-\tilde{r}_4} \tilde{n}_{\tilde{r}_1+i} \\
&= \frac{r_4(r_4-1)}{2} + (r_1+1)r_4 + \frac{(n'_1-r_4)(n'_1+1+r_4)}{2} + \sum\limits_{i=1}^{n'_1-r_4} \left(n_{r_1+1+i}+1\right).
\end{align*}
The box to the left of the origin of $M_1$ is filled and hence $n_{r_1+1}=-1$. Therefore
\begin{align*}
\tilde{d}_1
&= \frac{r_4(r_4-1)}{2} + r_1 r_4 + r_4 + \frac{(n'_1+1-r_4)(n'_1+2+r_4)}{2} -n_1'-1  +n'_1-r_4 + \sum\limits_{i=1}^{n'_1-r_4} n_{r_1+1+i} \\
&= \frac{r_4(r_4-1)}{2} + r_1 r_4  + \frac{(n'_1+1-r_4)(n'_1+2+r_4)}{2} + \sum\limits_{i=1}^{n'_1+1-r_4} n_{r_1+i}.
\end{align*}
Hence $\tilde{d}_1=d_1$ such that $(-1)^{\tilde{d}}=(-1)^{d}$. This ends the proof of Theorem \ref{thm:appA}.

\subsection{Proof of Lemma \ref{lem:GLP5}}\label{sec:ConjPart}
From the previous subsections, it should be clear that we can play with the Maya diagrams to obtain several results. For example, Lemma \ref{lem:GLP5} can be viewed as a special case of Theorem \ref{thm:appA}. 

\begin{proof}[Proof of Lemma \ref{lem:GLP5}]
	Take the following Maya diagrams,
	\begin{equation} \label{eq:specialM1M2}
	\begin{aligned}
	M_1: & \quad \left( n'_1,\dots,n'_{r_4} \mid \emptyset \right), \\
	M_2: & \quad \left( m'_1,\dots,m'_{r_3} \mid \emptyset \right),
	\end{aligned}
	\end{equation} 
	where $n'_{r_4}\neq 0$ if $r_4 \geq 1$ and $m'_{r_3}\neq 0$ if $r_3 \geq 0$.
	In this situation we have that
	\begin{equation*}
	\Omega^{(\alpha)}_{M_1, M_2}(x) 
	= e^{-r_4x} x^{\alpha(r_3+r_4)} \cdot \Wr[f_1,\dots,f_r] 
	\end{equation*}
	where $r=r_3+r_4$ and $f_1, \ldots, f_r$  are as in \eqref{app:fj3}-\eqref{app:fj4} 
	with $r_1=r_2=0$. Using \eqref{eq:Wr1} with $g(x)= x^{-\alpha}$, 
	and comparing with \eqref{eq:OmegaLaMu},  we get that
	\begin{equation} \label{eq:conj0}
	\Omega^{(\alpha)}_{M_1, M_2}
	= \Omega^{(-\alpha)}_{\mu',\lambda'}
	\end{equation}
	where $\mu'$ and $\lambda'$ are the conjugate partitions. 
	We apply Theorem \ref{thm:appA} to the left-hand side, where we note that
	all double products are empty since $r_1=r_2=0$, and at the
	same time use Lemma \ref{lem:GLP4} to interchange the two partitions
	in the right-hand side of \eqref{eq:conj0}. We end up with 
	\begin{equation}\label{eq:conj1}
	(-1)^{d_1+d_2} \Omega^{(\alpha-t)}_{\lambda,\mu}(x) =
	(-1)^{\frac{r_3(r_3-1)}{2} + \frac{r_4(r_4-1)}{2}} 
	\Omega^{(-\alpha)}_{\lambda',\mu'}(-x)
	\end{equation}
	In this special situation, we have 
	\begin{align*}
	&d_1 = |\lambda| + \frac{r_4(r_4-1)}{2}, \\
	&d_2 = |\mu| + \frac{(m_1'+1-r_3)(m_1'-r_3)}{2}.
	\end{align*}
	The value $m_1'+1-r_3$ is the number of empty boxes to the left of the 
	origin of $M_2$ and therefore it is equal to the length of the partition $\mu$, 
	which we denote here by $r(\mu)$. Similarly, we have $r(\lambda)$
	as the length of the partition $\lambda$.  We also note that $\mu_1 = r_3$,
	and then \eqref{eq:conj1} reduces to 
	\begin{equation}\label{eq:conj2}
	\Omega^{(\alpha-t)}_{\lambda,\mu}(x)
	= (-1)^{\frac{\mu_1(\mu_1-1)}{2} + |\lambda|+|\mu| + 
		\frac{r(\mu)(r(\mu)-1)}{2}} 
	\Omega^{(-\alpha)}_{\lambda',\mu'}(-x).
	\end{equation}
	This is the result of Lemma \ref{lem:GLP5} if we replace $\alpha$ by $\alpha+t$.
	
	However, it remains  to check that the value of $t$ used in Lemma \ref{lem:GLP5} 
	agrees  with the one from Theorem \ref{thm:appA} for the special 
	case \eqref{eq:specialM1M2}.
	In Theorem \ref{thm:appA} we have $t = t_1 + t_2$ with $t_1 = t(M_1)$ and $t_2 = t(M_2)$.
	For the Maya diagrams \eqref{eq:specialM1M2} this is $t_1 = n_1'+1$ and 
	$t_2 = m_1'+1$.
	In Lemma \ref{lem:GLP5} we have $ t= \lambda_1 + \mu_1 + r(\lambda) + r(\mu)$.
	and indeed, it holds true that
	$\lambda_1 + r(\lambda) = n_1'+1$ and $\mu_1 + r(\mu) = m_1'+1$.
	This completes the proof of the lemma.
\end{proof}

\section{Construction of the exceptional Laguerre polynomial}\label{sec:ConstructionXLP}
Exceptional Laguerre polynomials are often denoted as $X_m$-Laguerre polynomials. In this section we prove that the exceptional Laguerre polynomial defined in \eqref{def:XLP1} using two partitions captures the most general format and we give explicit identities with respect to the definitions of $X_m$-Laguerre polynomials.

\subsection{Exceptional Laguerre polynomial revisited}
The exceptional Laguerre polynomial is obtained by adding one eigenfunction to the Wronskian of the generalized Laguerre polynomial. We defined two kinds of exceptional Laguerre polynomials in \eqref{def:XLP1} and \eqref{def:XLP2}. They are related by \eqref{eq:XLP1+2} and therefore we only considered \eqref{def:XLP1}. In general, we could take $f_1,\dots,f_r$ as in \eqref{app:fj1}-\eqref{app:fj4} and add one more function into the Wronskian where this function is one of the four types. By the reduction procedure, we can always reduce each set up to definition \eqref{def:XLP1} as stated in the following proposition. Therefore our definition of the exceptional Laguerre polynomial \eqref{def:XLP1} covered all possibilities. Recall that $\lambda'$ is the conjugated partition of $\lambda$.

\begin{proposition}
Let $f_1,\dots,f_r$ as in \eqref{app:fj1}-\eqref{app:fj4} and take two Maya diagrams $M_1,M_2$ as in \eqref{eq:M1M2}. Let $\lambda = \lambda(M_1)$ and $\mu = \lambda(M_2)$ be the partitions that are associated with these Maya diagrams and let $r(\lambda)$ and $r(\mu)$ denote the length of these partitions. Let $t_1 = t(M_1)$, $t_2 = t(M_2)$ as in \eqref{eq:tM}. Take $s\geq0$.
	\begin{enumerate}
		\item[\rm (a)] If $s\neq n_i$ for $i=1,\dots,r_1$, then $n:=s+t_1+|\lambda|+|\mu|-r(\lambda)\in\mathbb{N}_{\lambda,\mu}$ and 
		\begin{equation*}
			e^{-(r_2+r_4)x}x^{(\alpha+r_1+1+r_2)(r_3+r_4)} \cdot \Wr[f_1,\dots,f_r,L^{(\alpha)}_{s}(x)](x)
				= C_1 L^{(\alpha-t_1-t_2)}_{\lambda,\mu,n}(x),
		\end{equation*}
		for some constant $C_1$.
		
		\item[\rm (b)] If $s\neq m_i$ for $i=1,\dots,r_2$, then $n:=s+t_2+|\lambda|+|\mu|-r(\mu)\in\mathbb{N}_{\mu,\lambda}$ and
		\begin{equation*}
			e^{-(r_2+1+r_4)x}x^{(\alpha+r_1+r_2+1)(r_3+r_4)} \cdot \Wr[f_1,\dots,f_r,e^{x} L^{(\alpha)}_{s}(-x)](x)
				= C_2 L^{(\alpha-t_1-t_2)}_{\mu,\lambda,n}(-x),
		\end{equation*}
		for some constant $C_2$.
		
		\item[\rm (c)] If $s\neq m'_i$ for $i=1,\dots,r_3$, then $n:=s+m_1+1+|\lambda|+|\mu|-r(\mu')\in\mathbb{N}_{\mu',\lambda'}$ and
		\begin{equation*}
		e^{-(r_2+r_4)x}x^{(\alpha+r_1+r_2)(r_3+1+r_4)} \cdot \Wr[f_1,\dots,f_r,x^{-\alpha} L^{(-\alpha)}_{s}(x)](x)
		\\ = C_3 L^{(-\alpha')}_{\mu',\lambda',n}(x),
		\end{equation*}
		where $\alpha'=\alpha+n_1+m_1+2$ and for some constant $C_3$.
		
		\item[\rm (d)] If $s\neq n'_i$ for $i=1,\dots,r_4$, then $n:=s+n_1+1+|\lambda|+|\mu|-r(\lambda')\in\mathbb{N}_{\lambda',\mu'}$, then 
		\begin{equation*}
			e^{-(r_2+r_4+1)x}x^{(\alpha+r_1+r_2)(r_3+r_4+1)} \cdot \Wr[f_1,\dots,f_r,e^{x} x^{-\alpha} L^{(-\alpha)}_{s}(x)](x)
				= C_4 L^{(-\alpha')}_{\lambda',\mu',n}(-x),
		\end{equation*}
		where $\alpha'=\alpha+n_1+m_1+2$ and for some constant $C_4$.
	\end{enumerate}
\end{proposition}
\begin{proof}
	(a) By Theorem \ref{thm:appA}, we can shift the origin in both Maya diagrams to its canonical choice corresponding to both partitions $\lambda$ and $\mu$. By this process, $L^{(\alpha)}_{s}(x)$ shifts to $L^{(\alpha-t_1-t_2)}_{s+t_1}(x)$. Therefore we end up with definition \eqref{def:XLP1} and the degree is given by $n=s+t_1+|\lambda|+|\mu|-r(\lambda)\in\mathbb{N}_{\lambda,\mu}$. Note that the two conditions of $n\in\mathbb{N}_{\lambda,\mu}$ reduce to $s+t_1\geq 0$ and $s \neq n_i$ for $i=1,\dots,r_1$. The second condition is fulfilled by assumption and the first condition is trivially true by construction.
	
	(b) We apply the same procedure as in (a). In this case, the last function in the Wronskian $e^{x} L^{(\alpha)}_{s}(-x)$ shifts to $e^{x}L^{(\alpha-t_1-t_2)}_{s+t_2}(-x)$. Hence we end up with definition \eqref{def:XLP2} with degree $s+t_2+|\lambda|+|\mu|-r(\mu)\in\mathbb{N}_{\mu,\lambda}$. By \eqref{eq:XLP1+2}, interchanging the partitions gives definition \eqref{def:XLP1} up to a possible sign change. 
	
	(c) We shift both origins of the Maya diagrams to the right in such a way that all boxes to the right are empty, and the first box to the left is filled: $n_1+1$ steps for the first Maya diagram and $m_1+1$ steps for the second. This choice is related to the conjugated partition. Set $\alpha'=\alpha+n_1+m_1+2$. Similar as in Theorem \ref{thm:appA}, up to a constant we arrive at
	\begin{equation*}
		e^{-(r(\lambda')+1)x}x^{\alpha'(r(\lambda')+r(\mu')+1)} \cdot
			\Wr \left[g_1,\dots,g_{r(\lambda')},g_{r(\lambda')+1},\dots, g_{r(\lambda')+r(\mu')},g\right]
	\end{equation*}
	where
	\begin{align*}
	&g_j(x) = e^{x} x^{-\alpha'}L^{(-\alpha')}_{\bar{n}_j}(-x), && j=1,\dots,r(\lambda'),\\
	&g_{r(\lambda')+j} = x^{-\alpha'}L^{(-\alpha')}_{\bar{m}_j}(x), && j=1,\dots,r(\mu')
	\end{align*}
	for some non-negative integers $\bar{n}_j,\bar{m}_j$ and
	\begin{equation*}
	g(x)=e^{x} x^{-\alpha'} L^{(-\alpha')}_{s+m_1+1}(x)
	\end{equation*}
	because the function $e^{x} x^{-\alpha} L^{(-\alpha)}_{s}(x)$ is shifted to $e^{x} x^{-\alpha'} L^{(-\alpha')}_{s+m_1+1}(x)$. Hence all functions in the Wronskian consist of a common factor $x^{-\alpha'}$ which cancels out by the factor in front of the Wronskian by \eqref{eq:Wr1}. Next we can interchange the functions $g_j$ such that we end up with definition \eqref{def:XLP1} where the degree is $s-m_1-1+|\lambda'|+|\mu'|-r(\mu')$ and the parameter is $-\alpha'$. The identity follows as $|\lambda'|=|\lambda|$. 
	
	(d) The steps are similar to (c), however we end up with definition \eqref{def:XLP2}. By \eqref{eq:XLP1+2}, the result follows.
\end{proof}

\subsection{$X_m$-Laguerre polynomials} \label{subsec:Xm}
In the literature, exceptional Laguerre polynomials are often denoted as 
$X_m$-Laguerre polynomials where $m$ is the number of exceptional degrees \cite{GomezUllate_Kamran_Milson-09,GomezUllate_Kamran_Milson-10,GomezUllate_Kamran_Milson-12,GomezUllate_Marcellan_Milson,Grandati_Quesne,Ho,Ho_Odake_Sasaki,Ho_Sasaki,Horvath,Liaw_etal,Marquette_Quesne,Odake_Sasaki-10,Quesne,Sasaki_Tsujimoto_Zhedanov}. In our set-up, this number of exceptional degrees is given by 
\begin{equation}\label{eq:NumExcDeg}
m=|\mathbb{N}_{0} \setminus \mathbb{N}_{\lambda,\mu}|=|\lambda|+|\mu|,
\end{equation}
where $\mathbb{N}_{\lambda,\mu}$ is the degree sequence \eqref{eq:nLaMu}. 
Hence to obtain exactly $m$ exceptional degrees, we have to take the partitions 
$\lambda$ and $\mu$ in such a way that \eqref{eq:NumExcDeg} is satisfied. 
A few of these possibilities are studied in detail and referred as type I, 
type II and type III exceptional Laguerre polynomials. 
In \cite{Liaw_etal}, these polynomials are denoted by 
$L_{m,n}^{I,\alpha}, L_{m,n}^{II,\alpha}, L_{m,n}^{III,\alpha}$ where 
$m$ is the number of exceptional degrees, $n$ is the degree of the polynomial, 
$\alpha$ is a parameter and the Roman number relates to the corresponding type. 
The definitions of these polynomials are given below. As these 3 types only 
cover a small part of all exceptional Laguerre polynomials, we do not prefer 
this terminology and notation. The approach with partitions is more general
and covers all cases, as we have shown. 
 
We show how the $L_{m,n}^{I,\alpha}, L_{m,n}^{II,\alpha}, L_{m,n}^{III,\alpha}$
exceptional Laguerre polynomials are expressed in terms of partitions.

\begin{proposition} \label{prop:types}
Let $m\geq0$ be fixed. We have the following identities,
\begin{align} \label{eq:LmnI}
	L_{m,n}^{I,\alpha}		& = - L^{(\alpha-1)}_{\emptyset,(m),n},  && n \geq m,\\
	\label{eq:LmnII}
	L_{m,n}^{II,\alpha}	& = (-1)^{\frac{m(m+3)}{2}}(2m-n-\alpha-1) \cdot L^{(\alpha-m)}_{\emptyset,(1,\dots,1),n}, && n \geq m,\\
	\label{eq:LmnIII}
	L_{m,n}^{III,\alpha}	& = -n \cdot L^{(\alpha-m)}_{(1,\dots,1),\emptyset,n}, && n > m.
\end{align}
\end{proposition}
The partition $(1, \ldots, 1)$ in \eqref{eq:LmnII} and \eqref{eq:LmnIII} 
consists of $m$ values $1$. Its Young diagram consists of one single column of length $m$.
The partition $(m)$ in \eqref{eq:LmnI} is its conjugate partition, and its
Young diagram consists of one row of length $m$.

The identities of Proposition \ref{prop:types}
are derived below, see \eqref{eq:TypeI}, \eqref{eq:TypeII} and \eqref{eq:TypeIII}, 
and were verified with Maple for the examples given in the 
appendices of \cite{Ho_Sasaki,Liaw_etal}. For the rest of this subsection, let $m$ be fixed.

\paragraph{Type I exceptional Laguerre polynomial.}
The polynomial is given by
\begin{equation*}
L_{m,n}^{I,\alpha}(x)
:= \begin{vmatrix}
L_m^{(\alpha)}(-x)		&-L^{(\alpha)}_{n-m-1}(x) \\
L_m^{(\alpha-1)}(-x) 	&L_{n-m}^{(\alpha-1)}(x) 	
\end{vmatrix}
\end{equation*}
for $n\geq m$, see \cite[equation (3.2)]{Liaw_etal}. If we interchange the rows in this determinant and use the derivatives \eqref{app:eq:diff1} and \eqref{app:eq:diff2}, we obtain that
\begin{equation*}
L_{m,n}^{I,\alpha} 
= -e^{-x} \cdot \Wr \left[f_1, L_{n-m}^{(\alpha-1)}\right],
\end{equation*}
where
\begin{equation*}
f_1(x) = e^x L_m^{(\alpha-1)}(-x).
\end{equation*}
Hence by definition \eqref{def:XLP1},
\begin{equation}\label{eq:TypeI}
L_{m,n}^{I,\alpha}
= - L^{(\alpha-1)}_{\emptyset,(m),n},
\end{equation}
where $\lambda=\emptyset$ and $\mu=(m)$. By Lemma \ref{lem:ELP}, these type I $X_m$-Laguerre polynomials form a complete set on the positive real line with respect to the measure $\frac{x^{\alpha} e^{-x}}{\left(\Omega^{(\alpha-1)}_{\emptyset,(m)}\right)^2}$ if $\alpha-1>-1$. Moreover the generalized Laguerre polynomial is related to a classical Laguerre polynomial, i.e. 
\begin{equation*}
\Omega^{(\alpha-1)}_{\emptyset,(m)}(x) = L^{(\alpha-1)}_m(-x)
\end{equation*}
for all $x\in\mathbb{C}$ which follows directly by definition \eqref{eq:OmegaLaMu}.

\paragraph{Type II exceptional Laguerre polynomial.}
This polynomial is defined as
\begin{equation*}
L_{m,n}^{II,\alpha}(x)
:= \begin{vmatrix}
x L_m^{(-\alpha-1)}(x)	& -L^{(\alpha+1)}_{n-m}(x) \\
(m-\alpha-1) L_m^{(-\alpha-2)}(x) & L_{n-m-1}^{(\alpha+2)}(x) 	
\end{vmatrix}
\end{equation*}
for $n\geq m$, see \cite[Section 4]{Liaw_etal}. If we use \eqref{app:eq:diff1} and \eqref{app:eq:diff3}, we can rewrite this polynomial as
\begin{equation}\label{eq:TypeII0}
L_{m,n}^{II,\alpha}
= -x^{\alpha+2} \cdot \Wr \left[f_1, f_2\right],
\end{equation}
where
\begin{align*}
f_1(x) &= x^{-\alpha-1} L_m^{(-\alpha-1)}(x), \\
f_2(x) &= L^{(\alpha+1)}_{n-m}(x).
\end{align*}
The eigenfunction $f_2$ corresponds to \eqref{app:fj1} and $f_1$ to \eqref{app:fj3}. Now the idea is that we apply theorem \ref{thm:appA} to transfer the eigenfunction $f_2$ to a set of eigenfunctions related to \eqref{app:fj2}. In fact, we transpose the second Maya diagram to its canonical choice \eqref{eq:lambdaM}, i.e. $M_2$: $(m|\emptyset) \to (\emptyset|m,m-1,\dots,1)$. This process is obtained by shifting the origin $m+1$ steps to the left and therefore the parameter $\alpha+1$ decreases by $m+1$. We also get extra factors while doing this steps. Encode the first Maya diagram $M_1: (\emptyset | n-m)$ which is already in the canonical choice, we get that
\begin{align}
L_{m,n}^{II,\alpha}
&= \Omega^{(\alpha+1)}_{M_1,M_2} \nonumber \\
&= (-1)^{\frac{m(m+1)}{2}}(2m-n-\alpha-1) \cdot \Omega^{(\alpha-m)}_{\lambda(M_1),\lambda(M_2)},
\label{eq:TypeII1}
\end{align}
where the first equality follows by interchanging $f_1$ and $f_2$ in \eqref{eq:TypeII0} and the second equality is obtained by applying Theorem \ref{thm:appA}. As a final step, we want the first function in the Wronskian of $\Omega^{(\alpha+1)}_{\lambda(M_1),\lambda(M_2)}$ to be the last eigenfunction such that we obtain the same order as in definition \eqref{def:XLP1}. Therefore we interchange the functions in the Wronskian which leads to an extra factor $(-1)^m$, i.e.
\begin{equation}\label{eq:TypeII2}
\Omega^{(\alpha-m)}_{\lambda(M_1),\lambda(M_2)}
= (-1)^m \cdot L^{(\alpha-m)}_{\emptyset,(1,\dots,1),n}
\end{equation}
where the length of the partition $\mu$ is given by $m$ such that $r=m$. If we combine \eqref{eq:TypeII1} and \eqref{eq:TypeII2} we end up with
\begin{equation}\label{eq:TypeII}
L_{m,n}^{II,\alpha}
= (-1)^{\frac{m(m+3)}{2}}(2m-n-\alpha-1) \cdot L^{(\alpha-m)}_{\emptyset,(1,\dots,1),n}.
\end{equation}
Hence by Lemma \ref{lem:ELP}, these type II $X_m$-Laguerre polynomials form an orthogonal complete set on the positive real line with respect to the measure $\frac{x^{\alpha} e^{-x}}{\left(\Omega^{(\alpha-m)}_{\emptyset,(1,\dots,1)}\right)^2}$ if $\alpha-m>-1$. The generalized Laguerre polynomial can be written as a Laguerre polynomial, i.e.
\begin{equation}\label{eq:TypeII3}
\Omega^{(\alpha-m)}_{\emptyset,(1,\dots,1)}
\equiv (-1)^{\frac{m(m+1)}{2}} L_m^{(-\alpha-1)}
\end{equation}
which follows by Theorem \ref{thm:appA}.

\paragraph{Type III exceptional Laguerre polynomial.}
The type III exceptional Laguerre polynomial is obtained in \cite[Section 5]{Liaw_etal} and is defined as
\begin{equation*}
L_{m,n}^{III,\alpha}(x)
:= \begin{vmatrix}
x L_m^{(-\alpha-1)}(-x)	& -L^{(\alpha+1)}_{n-m-1}(x) \\
(m+1) L_m^{(-\alpha-2)}(-x) & L_{n-m-2}^{(\alpha+2)}(x) 	
\end{vmatrix}
\end{equation*}
for $n>m$ while for $n=0$ this polynomial is defined as the constant function 1. We can use the derivatives \eqref{app:eq:diff1} and \eqref{app:eq:diff4} to obtain
\begin{equation*}
L_{m,n}^{III,\alpha}
= -x^{\alpha+2}e^{-x} \cdot \Wr \left[f_1, f_2\right],
\end{equation*}
for $n>m$ and where
\begin{align*}
f_1(x) &= x^{-\alpha-1}e^x L_m^{(-\alpha-1)}(-x),\\
f_2(x) &= L^{(\alpha+1)}_{n-m-1}(x).
\end{align*}
The function $f_1$ corresponds to \eqref{app:fj4}. Similar as in the Type II case, we shift the (first) Maya diagram to its canonical form by Theorem \ref{thm:appA}, i.e. $M_1$: $(m|n-m-1) \to (\emptyset|n,m,m-1,\dots,1)$. Hence we end up with only eigenfunctions of category \eqref{app:fj1}. As a last step, we transpose the first eigenfunction to the end of the Wronskian such that the order of the degrees is the same as in definition \eqref{def:XLP1}, i.e. $(n,m,m-1,\dots,1) \to (m,m-1,\dots,1,n)$. Combining all extra factors, we obtain the identity
\begin{equation}\label{eq:TypeIII}
L_{m,n}^{III,\alpha}
= -n \cdot L^{(\alpha-m)}_{(1,\dots,1),\emptyset,n},
\end{equation}
for $n>m$ and where the length of the partition $\lambda$ is given by $m$ such that $r=m$. Lemma \ref{lem:ELP} tells us that these type III $X_m$-Laguerre polynomials form an orthogonal complete set to the measure $\frac{x^{\alpha} e^{-x}}{\left(\Omega^{(\alpha-m)}_{(1,\dots,1),\emptyset}\right)^2}$ if $m$ is even and $\alpha-m>-1$. The generalized Laguerre polynomial can be written as a Laguerre polynomial, i.e.
\begin{equation*}
\Omega^{(\alpha-m)}_{(1,\dots,1),\emptyset}(x)
= (-1)^{m} L_m^{(-\alpha-1)}(-x)
\end{equation*}
for all $x\in\mathbb{C}$ which follows by Theorem \ref{thm:appA} or by applying identity \eqref{lem:GLP4Formula} on \eqref{eq:TypeII3}.


\section{Zeros of exceptional Laguerre polynomials: New results}\label{sec:Zeros-Results}
In this section we state our new results according to the zeros of exceptional Laguerre polynomials. From now on, we assume that $\alpha \in \mathbb R$ and the two 
partitions $\lambda$ and $\mu$ are fixed. As before, the length of the 
partition $\lambda$ is $r_1$ and for $\mu$ it is $r_{2}$. Let $r=r_1+r_2$.
Throughout, we also use $(n_1, \ldots, n_{r_1})$ and $(m_1, \ldots, m_{r_2})$
as in \eqref{eq:lambdamu}. Our definition of exceptional Laguerre polynomials
generalizes the $X_m$-Laguerre polynomials considered in \cite{GomezUllate_Marcellan_Milson,Liaw_etal}
and our results on zeros are generalizations of some
of the results in these papers, see \cite[Section 3-4]{GomezUllate_Marcellan_Milson} and \cite[Section 5.5]{Liaw_etal}.

\subsection{Number of positive real zeros}
Under the conditions of Lemma \ref{lem:ELP}, the polynomials are a complete 
orthogonal system, and they are also eigenfunctions for a Sturm-Liouville problem
on $[0,\infty)$. 
From general Sturm-Liouville theory,  see for example \cite{Simon}, 
the number of positive real zeros of $L^{(\alpha)}_{\lambda, \mu,n}$ is given by 
\begin{equation}\label{eq:NumRegZer}
	|\{m \in \mathbb{N_{\lambda,\mu}} : m<n\}|.
\end{equation} 
Moreover all these zeros are simple. Hence, the number of simple positive real zeros is non-decreasing and increases to infinity when $n$ tends to infinity. When we only consider a Wronskian of Laguerre polynomials, i.e. $r_1=0$ or $r_2=0$, then the number of positive real zeros is determined in \cite{GarciaFerrero_GomezUllate} for almost every $\alpha\in(-1,\infty)$. It is given by the alternating sum of the elements in the partition.

The value of $L^{(\alpha)}_{\lambda, \mu, n}$ at 
the origin can be computed as in a way similar to the generalized Laguerre polynomial 
case. It is non-zero when $\alpha>-1$ which is assumed in Lemma \ref{lem:ELP}. 
Hence for $n$ large enough, there are $n-|\lambda|-|\mu|$ simple zeros in the 
orthogonality region $(0,\infty)$ and these zeros are called the regular zeros 
of $L^{(\alpha)}_{\lambda, \mu,n}$. We denote them by
\begin{equation*}
	0<x^{(\alpha)}_{1,n}<x^{(\alpha)}_{2,n}<\cdots<x^{(\alpha)}_{n-|\lambda|-|\mu|,n}.
\end{equation*}	
The remaining $|\lambda| + |\mu|$ zeros are in $\mathbb C \setminus [0,\infty)$ and 
these are the exceptional zeros. We denote them by
\begin{equation*}
	z^{(\alpha)}_{1,n},z^{(\alpha)}_{2,n},\dots,z^{(\alpha)}_{|\lambda|+|\mu|,n}.
\end{equation*}
If the conditions of Lemma \ref{lem:ELP} are not satisfied, the number of positive 
real zeros cannot be determined by Sturm-Liouville theory. It is even possible that 
there is a zero at the origin. Therefore we make the following definition.

\begin{definition}
For $n \in \mathbb N_{\lambda,\mu}$, we use $N(n)$ to denote the number of zeros 
of $L^{(\alpha)}_{\lambda, \mu,n}$ in  $(0,\infty)$ including multiplicity. We call these zeros the 
regular zeros of $L^{(\alpha)}_{\lambda, \mu,n}$. The remaining $n-N(n)$ zeros 
are called the exceptional zeros.
\end{definition}

We omitted the partitions $\lambda,\mu$ and the parameter $\alpha$ in our notation 
of $N(n)$ because it should be clear what they are.  Obviously we have that 
$0 \leq N(n)\leq n$. As stated before, under the conditions of Lemma \ref{lem:ELP}, 
the number of regular zeros $N(n)$ is given by \eqref{eq:NumRegZer}. 

A first result is that we have a lower bound for $N(n)$ and moreover the number of simple regular zeros tends to infinity as the degree $n$ tends to infinity.

\begin{theorem} \label{thm:XLPN(n)}
Let $\alpha + r > - 1$. Then for $n\in\mathbb{N}_{\lambda,\mu}$, we have 
\begin{equation} \label{eq:XLPN(n)}
	n-2(|\lambda|+|\mu|)-r_2 \leq N(n).
\end{equation}
Moreover, the number of simple positive real zeros of $ L^{(\alpha)}_{\lambda, \mu,n}$ increases to infinity as $n$ tends to infinity.
\end{theorem}

Now we are able to state our asymptotic results. These results justify the 
conjecture \cite[Conjecture 1.1]{Kuijlaars_Milson} for exceptional Laguerre polynomials. 
We prove that the regular zeros of the exceptional Laguerre polynomial have the same
 asymptotic behavior as the zeros of their classical counterpart and the exceptional 
zeros converge to the (simple) zeros of the generalized Laguerre polynomial.

\subsection{Mehler-Heine asymptotics}
For Laguerre polynomials we have the Mehler-Heine asymptotics 
for all $\alpha\in\mathbb{R}$, see Theorem \ref{thm:LagMehlerHeine} below. 
This result can be generalized to exceptional Laguerre polynomials. 
We use $J_{\nu}$ to denote the Bessel function of the first kind of order 
$\nu\in\mathbb{R}$ \cite{Watson}.

\begin{theorem} \label{thm:XLPMehlerHeine}
Take $\alpha\in\mathbb{R}$, then one has
\begin{equation}\label{eq:XLPMehlerHeine}
	\lim_{n \to \infty} \frac{(-1)^r}{n^{\alpha+r}} L^{(\alpha)}_{\lambda, \mu,n} \left( \frac{x}{4n} \right)
	= \Omega^{(\alpha)}_{\lambda,\mu}(0) 2^{\alpha+r} x^{- \frac{\alpha + r}{2}} J_{\alpha + r}(\sqrt{x}),
\end{equation}
uniformly for $x$ in compact subsets of the complex plane.
\end{theorem}

The function $x^{- \frac{\alpha + r}{2}} J_{\alpha + r}(\sqrt{x})$ is an entire 
function in the complex plane with an infinite number of zeros on the positive real line
in case $\alpha + r > -1$, and no other zeros. All zeros are simple. 
Therefore, if we apply Hurwitz theorem 
\cite[Theorem 1.91.3]{Szego}, we obtain the following convergence
property for the regular zeros.

\begin{corollary} \label{cor:XLPRegularZeros}
Assume $\alpha +r>-1$. For a positive integer $k$ and $n \in \mathbb N_{\lambda,\mu}$
large enough, let $x_{k,n}^{(\alpha)}$ denote the $k$th positive real zero 
of $L_{\lambda,\mu,n}^{(\alpha)}$, see also Theorem \ref{thm:XLPN(n)}. 
If $\Omega^{(\alpha)}_{\lambda,\mu}(0)\neq 0$, then we have
\begin{equation*}
	\lim_{n \to \infty} \sqrt{4n x^{(\alpha)}_{k,n}} = j_{\alpha+r,k},
\end{equation*}
where $j_{\alpha+r,k}$ is the $k$th positive zero of the Bessel function $J_{\alpha+r}$.
\end{corollary}

\subsection{Weak macroscopic limit of the regular zeros}
Whenever $\alpha+r>-1$ the number of regular zeros $N(n)$ tends to infinity 
as $n$ tends to infinity by Theorem \ref{thm:XLPN(n)}. The weak scaling limit of the 
zero-counting measure  of the regular zeros is the Marchenko-Pastur distribution 
$\frac{1}{2\pi} \sqrt{\frac{4-x}{x}} dx$. This is a generalization 
of its classical counterpart, see Theorem \ref{thm:LagMarchenkoPastur} below.

\begin{theorem} \label{thm:XLPMarchenkoPastur}
Take $\alpha\in\mathbb{R}$ such that $\alpha+r>-1$.	
Let $0<x^{(\alpha)}_{1,n}\leq \cdots \leq x^{(\alpha)}_{N(n),n}$ denote the regular 
zeros of the exceptional Laguerre polynomial $L^{(\alpha)}_{\lambda, \mu,n}$ 
where $n\in\mathbb{N}_{\lambda,\mu}$. Then for every bounded continuous function 
$f$ on the positive real line, 
\begin{equation} \label{eq:XLPMarchenkoPastur}
	\lim_{n\to\infty}\frac{1}{N(n)}\sum_{j=1}^{N(n)}f\left(\frac{x^{(\alpha)}_{j,n}}{N(n)}\right)
	=\frac{1}{2\pi}\int_{0}^{4}f(x)\sqrt{\frac{4-x}{x}}dx.
\end{equation}
\end{theorem}

\subsection{Convergence of the exceptional zeros}
If $\alpha+r>-1$, then the exceptional zeros are attracted by simple zeros of the generalized Laguerre polynomials. We use $z_1, \ldots, z_{|\lambda|+|\mu|}$
to denote the zeros of the generalized Laguerre polynomial $\Omega^{(\alpha)}_{\lambda,\mu}$.

\begin{theorem} \label{thm:XLPExceptionalZeros}
Take $\alpha\in\mathbb{R}$ such that $\alpha+r>-1$. Let $z_j$ be a simple zero of the
generalized Laguerre polynomial $\Omega^{(\alpha)}_{\lambda,\mu}$ where 
$z_j\in\mathbb{C}\setminus[0,\infty)$. Then this zero $z_j$ attracts an exceptional zero 
of the exceptional Laguerre polynomial $L^{(\alpha)}_{\lambda,\mu,n}$ as $n$ tends to
infinity at a rate $O(n^{-1/2})$.
That is, for $n$ large enough, we have 
\begin{equation} \label{eq:XLPExceptionalZeros}
	\min_{k = 1, \ldots, n-N(n)} 
	\left| z_j-z^{(\alpha)}_{k,n} \right| < \frac{c}{\sqrt{n}},
\end{equation}
for some positive constant $c$ and where $z^{(\alpha)}_{1,n},\dots,z^{(\alpha)}_{n-N(n),n}$
denote the exceptional zeros of the exceptional Laguerre polynomial 
$L^{(\alpha)}_{\lambda,\mu,n}$ with $n\in\mathbb{N}_{\lambda,\mu}$.
\end{theorem}

\begin{figure}[t]
	\centering
	\includegraphics[totalheight=8cm, trim=2.25cm 17.5cm 8.75cm 2cm,clip=true]{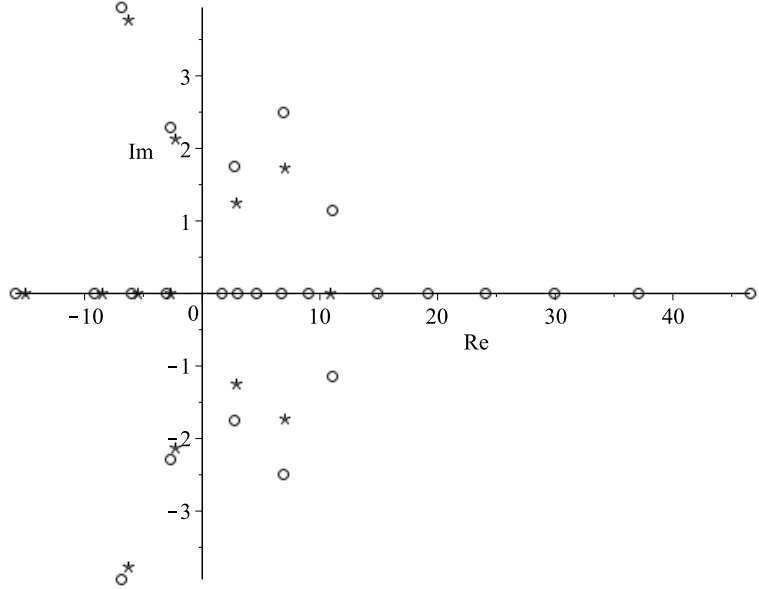}
	\caption{Zeros of the generalized (stars) and exceptional (open circles) Laguerre polynomial associated with $\lambda=(3,2)$, $\mu=(4,2,2)$, $\alpha=1$ and $n=25$.}
	\label{fig:1}
\end{figure}

In the situation of Lemma \ref{lem:ELP} we have that 
exceptional Laguerre polynomials form a complete set of orthogonal polynomials. 
Then we know that  $N(n) = n - |\lambda|- |\mu|$, if $n$ is large enough, 
and that the zeros of the
generalized Laguerre polynomial $\Omega_{\lambda,\mu}^{(\alpha)}$ are 
all in $\mathbb C \setminus [0,\infty)$.  If each of these $|\lambda| + |\mu|$ 
zeros is simple, it follows from Theorem \ref{thm:XLPExceptionalZeros}
that, for large $n$, each zero $z_j$ attracts exactly one exceptional zero 
of $L^{(\alpha)}_{\lambda,\mu,n}$. 
Hence, we can relabel the zeros of the exceptional Laguerre polynomial in such a way that $z_{j,n}^{(\alpha)}$ is close to $z_j$ and
\begin{equation*}
	z_{j,n}^{(\alpha)} = z_j + O\left(\frac{1}{\sqrt{n}}\right) \qquad
	\text{ as } n \to \infty.
\end{equation*}

Our results are numerically verified. In Figure \ref{fig:1} we plotted the zeros 
where we set $\alpha=1$, $\lambda=(3,2)$ and $\mu=(4,2,2)$. The thirteen zeros of the generalized Laguerre polynomial, which are indicated by a star, are all simple. The open circles are the zeros of the corresponding exceptional Laguerre polynomial of degree $25$. Note that there is one zero of the generalized Laguerre polynomial on the positive real line. Hence, the conditions of Theorem \ref{thm:XLPExceptionalZeros} are not satisfied. Nevertheless, it seems that this positive real zero attracts two exceptional zeros.

\subsection{Conjecture of simple zeros}
We do not have a proof that the zeros of $\Omega_{\lambda,\mu}^{(\alpha)}$ are simple, 
but we offer this  as a conjecture, based on numerical evidence.
\begin{conjecture}\label{conjecture}
Take $\alpha>-1$ and let $\lambda$ be an even partition. Then the zeros of the generalized Laguerre polynomial $\Omega^{(\alpha)}_{\lambda,\mu}$ are simple. 
\end{conjecture}

In Theorem \ref{thm:XLPExceptionalZeros}, we stated that the simple zeros of the generalized Laguerre polynomial attract an exceptional zero of the exceptional Laguerre polynomial. This result is comparable to the Hermite case \cite[Theorem 2.3]{Kuijlaars_Milson}. For the Hermite case, it was conjectured by Veselov \cite[Section 6]{Felder_Hemery_Veselov} that all zeros are simple. More concretely, he conjectured that the zeros of a Wronskian of Hermite polynomials are all simple, except possibly at the origin. As already said, we expect a similar conjecture to be true in our case, see Conjecture \ref{conjecture}. However, an important difference with the Hermite case is that we need a condition for $\lambda,\mu$ and $\alpha$ to be satisfied. We require a condition because of the following examples:
\begin{align*}
&\Omega^{(5)}_{(3,1),\emptyset}(x) = \frac{1}{8}(x-6)^3(x-14), \\
&\Omega^{(-2)}_{(2,2),\emptyset}(x) = \frac{1}{12}x^4, \\
&\Omega^{(-\frac{7}{4})}_{(1),(2)}(x) = -\frac{1}{2}\left(x+\frac{3}{4}\right)^3, \\
&\Omega^{(-\frac{13}{4})}_{(3),(3)}(x) = -\frac{1}{36} \left(x^2+\frac{15}{16}\right)^3.	
\end{align*}
It is clear that each of these generalized Laguerre polynomials has a non-simple 
zero, either real or non-real. Numerical simulations seem to suggest that every time 
that the exceptional Laguerre polynomials form a complete orthogonal system, 
the corresponding generalized Laguerre polynomial has simple zeros. 
This is the case when the conditions in Conjecture \ref{conjecture} are satisfied. Moreover, the conjecture holds true for type II exceptional Laguerre polynomials, see \cite[Proposition 4.3]{GomezUllate_Marcellan_Milson}.

\section{Zeros of exceptional Laguerre polynomials: Proofs}\label{sec:Zeros-Proofs}
In this section we give the proofs of the new results which were stated in the previous section.

\subsection{Proof of the lower bound of the regular zeros}\label{sec:ProofN(n)}
In this section we prove Theorem \ref{thm:XLPN(n)}. The proof
is based on the following lemma.

\begin{lemma}\label{lem:XLPLinCom}
Let $n\in\mathbb{N}_{\lambda,\mu}$ and take $\alpha\in\mathbb{R}$. 
Then the exceptional Laguerre polynomial $L^{(\alpha)}_{\lambda,\mu,n}$ is 
a linear combination of the Laguerre polynomials 
$L^{(\alpha+r)}_{n}, L^{(\alpha+r)}_{n-1},\dots,L^{(\alpha+r)}_{n-t}$ where 
\[ t=2(|\lambda|+|\mu|)+r_2 \]
is independent of $n$. 
\end{lemma}
\begin{proof}
Set $s = n-|\lambda|-|\mu|+r_1$ and consider $L^{(\alpha)}_{\lambda,\mu,n}$ which 
is defined in \eqref{def:XLP1}. By expanding the Wronskian determinant along the 
last column, one obtains
\begin{equation}\label{eq:XLPdecomposition1}
	L^{(\alpha)}_{\lambda,\mu,n}(x) = 
	\sum_{j=0}^{r}Q_{j}(x)\frac{d^j}{dx^j}L^{(\alpha)}_{s}(x),
\end{equation}
where $Q_j$ is a polynomial such that 
\begin{equation} \label{eq:degQj}
	\deg Q_j \leq |\lambda|+|\mu|-r_1+\min\{j,r_1\}
\end{equation} 
which  is independent of $n$. This upper bound for the degree follows from 
Proposition \ref{prop:degdet}. Notice that $Q_0 = (-1)^{r_2} \Omega^{(\alpha+1)}_{\tilde{\lambda},\mu}$ where 
$\tilde{\lambda}_i=\lambda_i-1$ for $i=1,\dots,r_1$ and $Q_r=\Omega^{(\alpha)}_{\lambda,\mu}$. 
For both polynomials, the upper bound of the degree is attained because 
of Lemma \ref{lem:GLP1}.

The derivative of a Laguerre polynomial is again 
a Laguerre polynomial of lower degree, but with a shifted parameter, i.e.,
\begin{equation}\label{eq:Lagder}
	\frac{d^j}{dx^j} L^{(\alpha)}_s(x) = (-1)^j L^{(\alpha+j)}_{s-j}(x),
\end{equation}
see e.g. \cite[(5.1.14)]{Szego}. Thus \eqref{eq:XLPdecomposition1} becomes
\begin{equation}\label{eq:XLPdecomposition2}
	L^{(\alpha)}_{\lambda,\mu,n}(x) = \sum_{j=0}^{r}Q_{j}(x)(-1)^jL^{(\alpha+j)}_{s-j}(x).
\end{equation}
It is possible to express the Laguerre polynomial as a sum of other Laguerre polynomials 
with shifted parameters,
\begin{equation*}
	L^{(\alpha)}_n (x) = \sum_{k=0}^{l}(-1)^k \binom{l}{k} L^{(\alpha+l)}_{n-k}(x),
\end{equation*}
which holds for every positive integer $l$ and $\alpha\in\mathbb{R}$, see \cite[(5.1.13)]{Szego}.
Using this in   \eqref{eq:XLPdecomposition2} with $n=s-j$ and $l=r-j$, 
we obtain 
\begin{align} \nonumber
	L^{(\alpha)}_{\lambda,\mu,n}(x) & = 
	\sum_{j=0}^{r}Q_{j}(x)(-1)^j
	\sum_{k=0}^{r-j}(-1)^k \binom{s-j}{k} L^{(\alpha+r)}_{s-j-k}(x) \\
	& = \label{eq:LemmaN(n):Proof:3} 
	\sum_{j=0}^{r}\tilde{Q}_{j}(x)L^{(\alpha+r)}_{s-j}(x),
\end{align}
where $\tilde{Q}_j$ is a certain polynomial of degree 
\begin{equation} \label{eq:degQbound} 
	\deg \tilde{Q}_j \leq |\lambda|+|\mu|-r_1+\min\{j,r_1\}. 
	\end{equation}
Next we use the three term recurrence satisfied by the Laguerre polynomials
from which it follows that $x L_k^{(\alpha + r)}$ is a linear combination
of the polynomials of degrees $k+1$, $k$ and $k-1$. Using the recurrence repeatedly,
we see that $\tilde{Q}_j L_{s-j}^{(\alpha+r)}$ is a linear combination
of Laguerre polynomials $L^{(\alpha+r)}_k$ with $k$ in the range
\[ s-j - \deg \tilde{Q}_j \leq k \leq s-j + \deg \tilde{Q}_j. \]
For $j=0, \ldots, r$ we have by the definitions of $s$ and $t$ and
the degree bound \eqref{eq:degQbound} that
\[ s-j - \deg \tilde{Q}_j  \geq s- r - (|\lambda|+|\mu|) = n - t
	\]
and
\[ s-j + \deg \tilde{Q}_j \leq s +  (|\lambda|+|\mu|) - r_1 = n. \]
Thus  for each $j=0, \ldots, r$, we have that
$\tilde{Q}_j L_{s-j}^{(\alpha+r)}$ is a linear combination of $L_k^{(\alpha+r)}$ 
with $n-t \leq k \leq n$ and the lemma follows because of \eqref{eq:LemmaN(n):Proof:3}. 
\end{proof}

\begin{remark}
In \cite{Kuijlaars_Milson} there is an analogous result for exceptional
Hermite polynomials, see Lemma 4.1. However, the statement and proof of this 
result in \cite{Kuijlaars_Milson} contain a mistake. The 
exceptional Hermite polynomial of degree $n$ associated with a partition 
$\lambda$ is a linear
combination of the Hermite polynomials $H_n, \ldots, H_{n-t}$ where $t = 2|\lambda|$,
while it is stated in \cite{Kuijlaars_Milson} that $t= |\lambda| + r$,
where $r$ is the length of $\lambda$.
This mistake, however, does not affect the further results in \cite{Kuijlaars_Milson}.
\end{remark}

With Lemma \ref{lem:XLPLinCom}, we can prove  Theorem \ref{thm:XLPN(n)}.

\begin{proof}[Proof of Theorem \ref{thm:XLPN(n)}] 
From Lemma \ref{lem:XLPLinCom} and the orthogonality of the
polynomials $L_k^{(\alpha+r)}$ on $[0,\infty)$, which holds since $\alpha + r > -1$,
we obtain
\begin{equation} \label{eq:XLPN(n)proof1} 
	\int_0^{\infty} Q(x) L_{\lambda, \mu, n}^{(\alpha)}(x)  x^{\alpha+r} e^{-x} dx = 0,
\end{equation}
whenever $Q$ is a polynomial of degree $< n-t$ where the number $t$ is as in
Lemma \ref{lem:XLPLinCom}.
This forces $L_{\lambda, \mu, n}^{(\alpha)}$ to have at least 
$n-t$ zeros in $(0,\infty)$ with odd multiplicity. Otherwise we can
construct a polynomial of degree $< n-t$ in such a way that 
$Q L_{\lambda, \mu, n}^{(\alpha)}$ does not change sign in $(0,\infty)$
and this would contradict \eqref{eq:XLPN(n)proof1}. Thus
\[ N(n) \geq n-t = n - 2(|\lambda| + |\mu|) - r_2. \]
Moreover, the number of zeros with odd multiplicity is bounded by this number $n-t$. The number of zeros with multiplicity at least 3 is trivially bounded by $\frac{n}{3}$. Hence, the number of simple regular zeros is at least $\frac{2n}{3}-t$ and therefore tends to infinity as $n$ tends to infinity.
\end{proof}

There is another consequence of Lemma \ref{lem:XLPLinCom} that we state
here for future reference. 

\begin{corollary}\label{cor:BeardonDriver}
Take $\alpha\in\mathbb{R}$ such that $\alpha+r>-1$. Let $n\in\mathbb{N}_{\lambda,\mu}$ 
such that $n>2(|\lambda|+|\mu|)+r_2$. Let 
$0<a^{(\alpha+r)}_{1,n}<a^{(\alpha+r)}_{2,n}<\cdots<a^{(\alpha+r)}_{n,n}$ 
denote the real and simple zeros of the Laguerre polynomial $L^{(\alpha+r)}_{n}$. 
Then, at least $n-2(|\lambda|+|\mu|)-r_2$ intervals 
$(a^{(\alpha+r)}_{k,n},a^{(\alpha+r)}_{k+1,n})$, where $1\leq k<n$, contain a zero 
of the exceptional Laguerre polynomial $L^{(\alpha)}_{\lambda,\mu,n}$.
\end{corollary}
\begin{proof}
This follows from Lemma \ref{lem:XLPLinCom} as was shown by 
Beardon and Driver \cite[Theorem 3.2]{Beardon_Driver} for arbitrary orthogonal
polynomials on the real line.
\end{proof} 

The lower bound in Theorem \ref{thm:XLPN(n)} also follows immediately from Corollary \ref{cor:BeardonDriver}.

\subsection{Proof of Mehler-Heine asymptotics}
Mehler-Heine asymptotics describe the behavior of orthogonal polynomials near the edges of their support. For the Laguerre polynomials this reads as follows \cite[Theorem 8.1.3]{Szego}.

\begin{theorem}\label{thm:LagMehlerHeine}
Take $\alpha\in\mathbb{R}$, then one has
\begin{equation}\label{eq:LagMehlerHeine}
	\lim_{n\to\infty} \frac{1}{n^{\alpha}}L_n^{(\alpha)}\left(\frac{x}{4n}\right) = 2^{\alpha} x^{-\frac{\alpha}{2}} J_{\alpha}\left(\sqrt{x}\right),
\end{equation}
uniformly for $x$ in compact subsets of the complex plane.
\end{theorem}	

According to Theorem \ref{thm:XLPMehlerHeine} a similar asymptotic behavior holds 
for the exceptional Laguerre polynomials as we are now going to prove.

\begin{proof}[Proof of Theorem \ref{thm:XLPMehlerHeine}]
Set $s = n-|\lambda|-|\mu|+r_1$ and write $L^{(\alpha)}_{\lambda,\mu,n}(x)$ 
as in \eqref{eq:XLPdecomposition1}. Clearly, if we do the expansion of
\eqref{def:XLP1} along the last column, we find $Q_r = \Omega_{\lambda,\mu}^{(\alpha)}$,
and therefore 
\begin{equation}\label{eq:LagMehlerHeine1}
	L^{(\alpha)}_{\lambda,\mu,n}(x) = 
	\sum_{j=0}^{r-1}Q_{j}(x)\frac{d^j}{dx^j}L^{(\alpha)}_{s}(x) + 
	\Omega^{(\alpha)}_{\lambda,\mu}(x) \frac{d^r}{dx^r}L^{(\alpha)}_{s}(x),
\end{equation}
where $Q_j$ is a polynomial of degree at most $|\lambda|+|\mu|-r_1+\min\{j,r_1\}$ 
which does not depend on $n$.

The limit \eqref{eq:LagMehlerHeine} also holds if we replace
$L_n^{(\alpha)}$ by $L_s^{(\alpha)}$, where $s = n-c$ for some constant $c$.
Thus
\begin{equation}\label{eq:LagMehlerHeine2}
	\lim_{n\to\infty} \frac{1}{n^{\alpha}}L_s^{(\alpha)}\left(\frac{x}{4n}\right) 
	= 2^{\alpha} x^{-\frac{\alpha}{2}} J_{\alpha}\left(\sqrt{x}\right),
\end{equation}
uniformly for $x$ in compact subsets of the complex plane.	
Because of the uniform convergence, \eqref{eq:LagMehlerHeine2} can be differentiated 
with respect to $x$ any number of times. Hence, for every non-negative integer $j$,
\begin{equation} \label{eq:LagMehlerHeine3}
	\lim_{n\to\infty} \frac{1}{4^j n^{\alpha +j}}
	\left( \frac{d^{j}}{dx^{j}} L_s^{(\alpha)} \right) \left(\frac{x}{4n}\right)
	 = 
	 \frac{d^{j}}{dx^j}\left(2^{\alpha} x^{-\frac{\alpha}{2}} 
	 J_{\alpha}\left(\sqrt{x}\right)\right).
\end{equation}
In particular
\[ \left( \frac{d^{j}}{dx^{j}} L_s^{(\alpha)} \right) \left(\frac{x}{4n}\right)
	= O \left(n^{\alpha + j}\right) \]
as $n \to \infty$ uniformly for $x$ in compacts. In fact, the previous equality can also be obtained directly because the derivative of the Laguerre polynomials is again a Laguerre polynomial \eqref{eq:Lagder} and therefore we can apply \eqref{eq:LagMehlerHeine} with the correct parameter.

Hence, the limiting behavior of  
\[ \frac{1}{n^{\alpha+r}} L_{\lambda, \mu,n}^{(\alpha)} \left( \frac{x}{4n} \right)
\]
is determined by the last term in \eqref{eq:LagMehlerHeine1} only. We find from \eqref{eq:LagMehlerHeine1} and 
\eqref{eq:LagMehlerHeine3} with $j=r$,
\begin{equation*}
\lim_{n \to \infty} 
	\frac{1}{n^{\alpha +r}} L_{\lambda, \mu,n}^{(\alpha)} \left( \frac{x}{4n} \right)
	= \Omega_{\lambda, \mu}^{(\alpha)}(0) 4^r
	 \frac{d^{r}}{dx^r}\left(2^{\alpha} x^{-\frac{\alpha}{2}} 
	 J_{\alpha}\left(\sqrt{x}\right)\right).
\end{equation*}
Because of the identity for Bessel functions
\[ \frac{d}{dx} \left( x^{- \frac{\alpha}{2}} J_{\alpha}(\sqrt{x}) \right)
	= - \frac{1}{2} x^{-\frac{\alpha+1}{2}} J_{\alpha+1}(\sqrt{x}) \]
we obtain the desired result \eqref{eq:XLPMehlerHeine}.
\end{proof}

Next we apply Hurwitz's Theorem to obtain the asymptotic convergence for the 
regular zeros as stated in Corollary \ref{cor:XLPRegularZeros}. 
We need the assumption that 
$\Omega^{(\alpha)}_{\lambda,\mu}(0) \neq 0$ so that the right-hand side of 
\eqref{eq:XLPMehlerHeine} is not identically zero.

\begin{proof}[Proof of Corollary \ref{cor:XLPRegularZeros}]
Applying Hurwitz's theorem to \eqref{eq:XLPMehlerHeine} gives us that 
those zeros of $L^{(\alpha)}_{\lambda,\mu,n}(\frac{x}{4n})$ that do not 
tend to infinity, tend to the zeros of 
$x^{-(\alpha + r)/2} J_{\alpha+r}(\sqrt{x})$ as $n$ tends to 
infinity. All these limiting zeros are simple and lie on the positive
real line.  

Since the zeros are simple, Hurwitz's theorem also says that each 
zero of $J_{\alpha+r}(\sqrt{x})$  attracts exactly one  zero of 
$L^{(\alpha)}_{\lambda,\mu,n}(\frac{x}{4n})$ as $n \to \infty$. 
This zero has to be real for large enough $n$, since its complex conjugate is a 
zero as well and if it were not real
for large $n$ then two zeros of $L^{(\alpha)}_{\lambda,\mu,n}(\frac{x}{4n})$
would approach the same simple zero of  $J_{\alpha+r}(\sqrt{x})$.  
\end{proof}

\subsection{Proof of the weak macroscopic limit of the regular zeros}
In this section we prove that the limit behavior of the zero-counting measure 
of the regular zeros of the exceptional Laguerre polynomials is given by the 
Marchenko-Pastur distribution. It is a generalization of the following well-known 
limit  of the zero-counting measure for Laguerre polynomials, see e.g.\ 
\cite[Theorem 1]{Gawronski}. 

\begin{theorem}\label{thm:LagMarchenkoPastur}
Let $ 0 < a^{(\alpha)}_{1,n} < a^{(\alpha)}_{2,n} < \dots < a^{(\alpha)}_{n,n} < \infty$
denote the zeros of the Laguerre polynomial $L^{(\alpha)}_{n}$ where $\alpha>-1$. 
Then for any bounded continuous function $f :[0,\infty) \to \mathbb R$ it 
is true that
\begin{equation}\label{eq:LagMarchenkoPastur}
	\lim_{n\to\infty}\frac{1}{n}\sum_{j=1}^{n}f\left(\frac{a^{(\alpha)}_{j,n}}{n}\right)
	=\frac{1}{2\pi}\int_{0}^{4}f(x)\sqrt{\frac{4-x}{x}}dx.
\end{equation}
\end{theorem}

Knowing  Theorem \ref{thm:LagMarchenkoPastur}, we prove  
Theorem \ref{thm:XLPMarchenkoPastur} in essentially the 
same way as Theorem 2.2 in \cite{Kuijlaars_Milson}, which dealt with the
semicircle law for the scaled zeros of exceptional Hermite polynomials.

\begin{proof}[Proof of Theorem \ref{thm:XLPMarchenkoPastur}]
Suppose $\alpha+r>-1$. Take $n\in\mathbb{N}_{\lambda,\mu}$ such that 
$n>2(|\lambda|+|\mu|)+r_2$. Let 
$0<a^{(\alpha+r)}_{1,n}<a^{(\alpha+r)}_{2,n}<\cdots<a^{(\alpha+r)}_{n,n}$ 
denote the zeros of the  Laguerre polynomial $L^{(\alpha+r)}_{n}$. 
From Lemma \ref{lem:XLPLinCom} and Corollary \ref{cor:BeardonDriver} 
it follows that at least 
$n-2(|\lambda|+|\mu|)-r_2$ intervals $(a^{(\alpha+r)}_{j,n},a^{(\alpha+r)}_{j+1,n})$, 
where $1\leq j<n$, contain a zero of the exceptional Laguerre polynomial 
$L^{(\alpha)}_{\lambda,\mu,n}$.

For any choice of $\xi_{j,n}\in(a^{(\alpha+r)}_{j,n},a^{(\alpha+r)}_{j+1,n})$ 
for every $j=1,\dots,n-1$ we get that the limit \eqref{eq:LagMarchenkoPastur} 
is still satisfied, i.e.,
\begin{equation*}
	\lim_{n\to\infty}\frac{1}{n-1}\sum_{j=1}^{n-1}f\left(\frac{\xi_{j,n}}{n}\right)
	=\frac{1}{2\pi}\int_{0}^{4}f(x)\sqrt{\frac{4-x}{x}}dx.
\end{equation*}	

For $n$ large, we can take $\xi_{j,n}$ to be equal to a zero of of 
$L^{(\alpha)}_{\lambda,\mu,n}$ for at least $n-2(|\lambda|+|\mu|)-r_2$ 
values of $j$.
By dropping the other indices in the sum, the limit will not be affected as $f$ 
is bounded. Next, one can extend the sum by including the remaining positive real 
zeros of $L^{(\alpha)}_{\lambda,\mu,n}$ because their number
remains bounded as $n$ increases. Hence, we obtain \eqref{eq:XLPMarchenkoPastur}.
\end{proof}

\subsection{Proof of the convergence of the exceptional zeros}\label{sec:ProofExceptionalZeros}
In this section we prove Theorem \ref{thm:XLPExceptionalZeros} which deals with 
the convergence of the exceptional zeros. Recall the weight
\begin{equation} \label{eq:Wlamu} 
	W^{(\alpha)}_{\lambda, \mu}(x) =  \frac{x^{\alpha+r} e^{-x}}
	{\left(\Omega_{\lambda,\mu}^{(\alpha)}(x) \right)^2} 
	\end{equation}
that appeared in Lemma \ref{lem:ELP} for $x \in (0,\infty)$. 
In this section we view \eqref{eq:Wlamu} as a meromorphic function 
in $\mathbb C \setminus \{0\}$
with poles at the zeros of $\Omega_{\lambda,\mu}^{(\alpha)}$.
We also consider a general parameter $\alpha \in \mathbb R$ and partitions $\lambda$ and $\mu$. We will need the following property.
\begin{lemma} \label{lem:Residue}
Take $\alpha\in\mathbb{R}$. For every $n \in \mathbb N_{\lambda, \mu}$ we have that
\[ \left(L^{(\alpha)}_{\lambda,\mu,n} \right)^2 W^{(\alpha)}_{\lambda,\mu} \]
has zero residue at each of its poles in $\mathbb C \setminus \{0\}$.
\end{lemma}

\begin{proof}
We apply a Darboux-Crum transformation to the differential operator 
\eqref{eq:LagDV1} with eigenfunctions $\varphi_{n_1}^{(\alpha)}, \ldots,
\varphi_{n_{r_1}}^{(\alpha)}, \psi_{m_1}^{(\alpha)}, \ldots, \psi_{m_{r_2}}^{(\alpha)}$,
see Table \ref{tab:2} in Section \ref{sec:Laguerre}. It leads to a new differential operator
\begin{equation} \label{eq:Oplamu} 
	y  \mapsto - y'' + V_{\lambda,\mu} y 
	\end{equation}
with potential
\begin{equation} \label{eq:Vlamu}
	V_{\lambda,\mu}(x)
	= x^2 + \frac{4\alpha^2-1}{4x^2}  - 
	2 \frac{d^2}{dx^2} \log \left(\Wr[\varphi^{(\alpha)}_{n_1},\dots, 
	\varphi^{(\alpha)}_{n_{r_1}}, \psi^{(\alpha)}_{m_1},\dots, \psi^{(\alpha)}_{m_{r_2}}]\right).
\end{equation}
The differential operator \eqref{eq:Oplamu} has eigenfunctions of the form
\begin{equation} \label{eq:OplamuEF}
	\frac{\Wr[\varphi^{(\alpha)}_{n_1},\dots, \varphi^{(\alpha)}_{n_{r_1}}, 
	\psi^{(\alpha)}_{m_1},\dots, \psi^{(\alpha)}_{m_{r_2}},\varphi^{(\alpha)}_s]}
	{\Wr[\varphi^{(\alpha)}_{n_1},\dots, \varphi^{(\alpha)}_{n_{r_1}}, 
	\psi^{(\alpha)}_{m_1},\dots, \psi^{(\alpha)}_{m_{r_2}}]}, 
\end{equation} 
where $s\geq 0$ and $s\neq n_j$ for every $j=1, \ldots, r_1$. 

Using \eqref{eq:Wr1} and \eqref{eq:Wr2}, we can express
the Wronskian in \eqref{eq:Vlamu} as a Wronskian for the functions $f_1, \ldots, f_r$
from \eqref{eq:fj1}-\eqref{eq:fj2}, and it follows from \eqref{eq:OmegaLaMu} that
\[ 
\Wr[\varphi^{(\alpha)}_{n_1},\dots, 
	\varphi^{(\alpha)}_{n_{r_1}}, \psi^{(\alpha)}_{m_1},\dots, \psi^{(\alpha)}_{m_{r_2}}]
	= 2^{\frac{r(r-1)}{2}} x^{\alpha r + \frac{r^2}{2}} e^{-\frac{r}{2}x^2}
		\Omega^{(\alpha)}_{\lambda,\mu}(x^2). 
\]
Hence
\begin{equation} \label{eq:Vlamu2} 
	V_{\lambda,\mu}(x) 
	= x^2 + 2r + \frac{4(\alpha+r)^2-1}{4x^2}  - 
	2 \frac{d^2}{dx^2}\log \left( \Omega^{(\alpha)}_{\lambda,\mu}(x^2) \right).
\end{equation}
Similarly, the eigenfunction \eqref{eq:OplamuEF} can be written as
\begin{equation} \label{eq:OplamuEF2}
	2^r x^{\alpha+\frac{1}{2}+r} e^{-\frac{1}{2}x^2} 
	\frac{L^{(\alpha)}_{\lambda,\mu,n}(x^2)}{\Omega^{(\alpha)}_{\lambda,\mu}(x^2)}
\end{equation}
if we choose $n \in \mathbb N_{\lambda,\mu}$ and $s=n-|\lambda|-|\mu|+r_1$.

We now use the fact that the operator \eqref{eq:Oplamu} has trivial
monodromy at every point $p \in \mathbb C \setminus \{0\}$, see \cite[Proposition 5.21]{GarciaFerrero_GomezUllate_Milson}. This means that any
eigenfunction of \eqref{eq:Oplamu} is meromorphic around $p$.

From \eqref{eq:Vlamu2} we see that $p \neq 0$ is a pole of $V_{\lambda,\mu}$ 
if and only if $p^2$ is a zero of $\Omega^{(\alpha)}_{\lambda,\mu}$. 
If $p^2$ is a zero of order $d_p \geq 0$ then 
\[ V_{\lambda,\mu}(x) =  2 d_p (x-p)^{-2} + O(1) \qquad \text{ as } x \to p. \]
By Proposition 3.3 of \cite{Duistermaat_Grunbaum}
we then have that $2 d_p = \nu_p(\nu_p+1)$ for some positive integer $\nu_p$.
Moreover, every eigenfunction $f$ has a Laurent expansion around $p$ 
\[ f(x) = (x-p)^{-\nu_p} \sum_{j=0}^{\infty} f_j (x-p)^j \]
with $f_{2j-1} = 0$ for $j=1, \ldots, \nu_p$.
This property implies that $f^2$ has zero residue at $x=p$.
 
This in particular holds for the eigenfunction \eqref{eq:OplamuEF2}
and so 
\begin{equation}\label{eq:IntegralArg}
	x^{2\alpha + 1 + 2r} e^{-x^2} 
	\left( \frac{L^{(\alpha)}_{\lambda,\mu,n}(x^2)}{\Omega^{(\alpha)}_{\lambda,\mu}(x^2)} \right)^2 
	=  
	x \left( L^{(\alpha)}_{\lambda,\mu,n}(x^2) \right)^2  
	W_{\lambda,\mu}^{(\alpha)}(x^2)	
\end{equation}
has zero residue at each of its poles in $\mathbb C \setminus \{0\}$. Now let $p^2\in\mathbb{C}\setminus\{0\}$ be a pole of $\left(L^{(\alpha)}_{\lambda,\mu,n} \right)^2 W^{(\alpha)}_{\lambda,\mu}$. The residue is given by
\begin{equation} \label{eq:residueintegral}
	\frac{1}{2\pi i} \oint_{\gamma_{p^2}} 
	\left(L^{(\alpha)}_{\lambda,\mu,n}(z) \right)^2 W^{(\alpha)}_{\lambda,\mu}(z) \, dz,
\end{equation}
where $\gamma_{p^2}$ is a small circle going around $p^2$ in counterclockwise direction. 
By a change of variables $z \mapsto z^2$, we obtain that the 
residue is equal to
\begin{equation*}
	\frac{1}{\pi i} \oint_{\gamma_p} z 
	\left(L^{(\alpha)}_{\lambda,\mu,n}(z^2) \right)^2 W^{(\alpha)}_{\lambda,\mu}(z^2)dz
\end{equation*} 
with $\gamma_p$ around $p$, and $p$ is a pole of \eqref{eq:IntegralArg}. 
This integral is zero, since \eqref{eq:IntegralArg} has zero residues,
and we see that \eqref{eq:residueintegral} is zero.
The lemma follows.  
\end{proof}

Note that if we replace $\varphi^{(\alpha)}_s$ by $\psi^{(\alpha)}_s$ in \eqref{eq:OplamuEF} we obtain other eigenfunctions of the differential operator \eqref{eq:Oplamu}. These eigenfunctions can be written in terms of
$\tilde{L}^{(\alpha)}_{\lambda,\mu,n}(x^2)$ and $\Omega^{(\alpha)}_{\lambda,\mu}(x^2)$. Now we are able to prove the asymptotic behavior of the exceptional zeros.

\begin{proof}[Proof of Theorem \ref{thm:XLPExceptionalZeros}]

Let $z_j$ be simple zero of $\Omega^{(\alpha)}_{\lambda,\mu}$ where 
$z_j\in\mathbb{C}\setminus[0,\infty)$.
Then $z_j$ is a double pole of $W_{\lambda,\mu}^{(\alpha)}$ and by
 \eqref{lem:Residue}
\begin{equation} \label{eq:Taylorexpansion} 
	(x-z_j)^2 \left( L_{\lambda,\mu,n}^{(\alpha)}(x) \right)^2 
	W_{\lambda,\mu}^{(\alpha)}(x)
	= C_0 + C_1(x-z_j) + O\left( (x-z_j)^2 \right) \qquad \text{as } x \to z_j 
	\end{equation}
for a certain constant $C_0$ and $C_1 = 0$. 
We may assume $C_0 \neq 0$, otherwise $z_j$ is a zero of 
$L^{(\alpha)}_{\lambda,\mu,n}$ as well and then \eqref{eq:XLPExceptionalZeros} is clearly satisfied.

Since $C_0 \neq 0$, we can take an analytic logarithm of \eqref{eq:Taylorexpansion} 
in the neighborhood of $z_j$, and since $C_1=0$, its derivative
vanishes at $z_j$.
The logarithmic derivative of the left-hand side of \eqref{eq:Taylorexpansion} is
\begin{equation} \label{eq:logder} 
	\frac{2}{x-z_j} + 2 \sum_{k=1}^{N(n)} \frac{1}{x-x_{k,n}^{(\alpha)}}
	+ 2 \sum_{k=1}^{n-N(n)} \frac{1}{x-z_{k,n}^{(\alpha)}}
		+ \frac{\alpha+r}{x} - 1 
		- 2 \sum_{k=1}^{|\lambda|+|\mu|} \frac{1}{x-z_k} 
\end{equation}
where, as before,
$0 < x_{1,n}^{(\alpha)} \leq \cdots \leq x^{(\alpha)}_{N(n),n}$ 
are the regular zeros and $z^{(\alpha)}_{1,n},\dots,z^{(\alpha)}_{n-N(n),n}$
are the exceptional zeros of the exceptional Laguerre polynomial and
$z_1, \ldots, z_{|\lambda|+|\mu|}$ are the zeros of $\Omega_{\lambda,\mu}^{(\alpha)}$.
Thus \eqref{eq:logder} vanishes as $x \to z_j$, and it gives us the identity
\begin{equation}\label{eq:logder2}
	\sum_{k=1}^{N(n)}\frac{1}{z_j-x^{(\alpha)}_{k,n}}
	+ \sum_{k=1}^{n-N(n)}\frac{1}{z_j-z^{(\alpha)}_{k,n}}
	= \frac{1}{2}-\frac{\alpha+r}{2z_j} + 	
	 \sum_{\shortstack{$\scriptstyle k=1 $ \\ 
	$ \scriptstyle k\neq j$}}^{|\lambda| + |\mu|} 
	\frac{1}{z_j-z_k}.
\end{equation}
	
On the interval $[0,1]$, the number of zeros of $L^{(\alpha+r)}_n$ grows 
roughly like $c\sqrt{n}$ as $n$ tends to infinity for some positive constant $c$, 
see \cite[Theorem 6.31.3]{Szego}. Because of Corollary \ref{cor:BeardonDriver}
the same holds true 
for the number of zeros of $L^{(\alpha)}_{\lambda,\mu,n}$ in $[0,1]$. 
We now distinguish between $z_j$ being real and $z_j$ being non-real.
	
\noindent	
{\bf Case 1:} $\im(z_j)\neq 0$. \\
As the non-real roots of polynomials with real coefficients 
come in conjugate pairs, we may assume that $\im(z_j)>0$. 
The  imaginary part of the first sum in the left-hand side of \eqref{eq:logder2} simplifies to 
\begin{equation*}
	\im\left(\sum_{k=1}^{N(n)}\frac{1}{z_j-x^{(\alpha)}_{k,n}}\right)
	= -\sum_{k=1}^{N(n)}\frac{\im(z_j)}{|z_j-x^{(\alpha)}_{k,n}|^2}.
\end{equation*}
All terms in this sum have the same sign, namely the sign of $\im(z_j)$ 
which is positive. Therefore, by restricting the sum to the zeros 
which are in the  interval $[0,1]$ and using the fact that there
are at least $c\sqrt{n}$ such zeros, we have that for $n$ large enough
\begin{align*}
	\im\left(\sum_{k=1}^{N(n)}\frac{1}{z_j-x^{(\alpha)}_{k,n}}\right)
	&< -\sum_{k=1}^{\lfloor{c\sqrt{n}}\rfloor} \frac{\im(z_j)}{|z_j-x^{(\alpha)}_{k,n}|^2} 
	\\
	&<-c_1 \sqrt{n}
\end{align*}
for some constant $c_1 > 0$. To obtain the second inequality we used that
for $x^{(\alpha)}_{k,n} \in [0,1]$,
\begin{equation}\label{eq:bounddistance}
	|z_j-x^{(\alpha)}_{k,n}|^2 \leq \max\{|z_j-1|^2, |z_j|^2\},
\end{equation}
where the right-hand side is independent of $n$. 

The right-hand side of \eqref{eq:logder2} does not depend on $n$. Therefore, to balance the terms, one has that for $n$ large enough
\begin{equation*}
	\im\left(\sum_{k=1}^{n-N(n)}\frac{1}{z_j-z^{(\alpha)}_{k,n}}\right) > c_1 \sqrt{n}.
\end{equation*}
This is a finite sum and the number of terms is bounded by $2(|\lambda| + |\mu|) - r_2$
because of Theorem \ref{thm:XLPN(n)}. So there is at least one term that
is also of order $\sqrt{n}$. 
Hence, there is a constant $c_2 > 0$ such that for every large enough $n$ there
exists a zero $z^{(\alpha)}_{k,n}$ with 
\begin{equation}\label{eq:distance}
	\im\left(\frac{1}{z_j-z^{(\alpha)}_{k,n}}\right) > c_2 \sqrt{n}.
\end{equation}
The fact that $\im(\frac{1}{z})<\frac{1}{|z|}$ implies that from \eqref{eq:distance} we find
\begin{equation*}
	|z_j-z^{(\alpha)}_{k,n}|<\frac{1}{c_2\sqrt{n}},
\end{equation*}
which ends the proof in this case.

\noindent	
{\bf Case 2:} $\im(z_j)=0$. \\
Since $z_j \not\in [0,\infty)$, by assumption, 
we then have that  $z_j$ is an element of the negative real axis. 
We can give a similar argument as in Case 1, however one has to consider the real part of \eqref{eq:logder2} and \eqref{eq:bounddistance} needs to be replaced by 
\begin{equation*}
	z_j-x^{(\alpha)}_{l,n} \geq z_j -1.
\end{equation*}

The two cases complete the proof of inequality \eqref{eq:XLPExceptionalZeros} and the theorem follows.
\end{proof}
 
\section*{Acknowledgements}
Arno Kuijlaars is supported by long term structural funding-Methusalem grant of 
the Flemish Government, by the Belgian Interuniversity Attraction Pole P07/18, 
by KU Leuven Research Grant OT/12/073, and by FWO Flanders projects G.0934.13 and 
G.0864.16.

\end{document}